\documentclass[11 pt]{amsart}
\usepackage{amsmath}
\usepackage{latexsym,amsfonts,amssymb,mathrsfs}
\usepackage{mathdots}
\usepackage{color}
\usepackage [all,2cell,color]{xy}
\usepackage{thmtools}
\usepackage{enumitem} 
\setlist[enumerate]{labelsep=*, leftmargin=0.7cm}

\usepackage{stackengine} 

\stackMath

\newlist{remarklist}{enumerate}{1}
\setlist[remarklist]{label={(\arabic{remarklisti})}, ref={(\arabic{remarklisti})}}
\makeatletter
\renewcommand{\p@remarklisti}{\perh@ps{\thethm}}
\protected\def\perh@ps#1#2{\textup{#1#2}}
\newcommand{\itemrefperh@ps}[2]{\textup{#2}}
\newcommand{\itemref}[1]{\begingroup\let\perh@ps\itemrefperh@ps\ref{#1}\endgroup}
\makeatother

\usepackage{tikz-cd}
\usepackage[T1]{fontenc}
 \usepackage[utf8]{inputenc}

\usetikzlibrary{positioning} 
\usetikzlibrary{arrows, decorations.markings} 
\tikzset{arrow/.style={-stealth}}
\tikzset{arrowshorter/.style={-stealth, shorten <=2pt, shorten >=2pt}}

\theoremstyle{plain}   
\newtheorem{thm}{Theorem}[section] 
\makeatletter\let\c@thm\c@thm\makeatother

\makeatletter\let\c@cor\c@thm\makeatother
\newtheorem{lem}{Lemma}[section]
\makeatletter\let\c@lem\c@thm\makeatother
\newtheorem{prop}{Proposition}[section]
\makeatletter\let\c@prop\c@thm\makeatother

\makeatletter\let\c@claim\c@thm\makeatother

\makeatletter\let\c@conjecture\c@thm\makeatother

\newtheorem*{unnumberedconjecture}{Conjecture}

\newtheorem*{unnumberedtheoremA}{Theorem A}
\newtheorem*{unnumberedtheoremB}{Theorem B}
\newtheorem*{unnumberedtheoremC}{Theorem C}

\theoremstyle{definition}

\newtheorem{defn}{Definition}[section]
\makeatletter\let\c@defn\c@thm\makeatother

\makeatletter\let\c@const\c@thm\makeatother
\newtheorem{notn}{Notation}[section]
\makeatletter\let\c@notn\c@thm\makeatother

\theoremstyle{remark}

\newtheorem{rmk}{Remark}[section]
\makeatletter\let\c@rmk\c@thm\makeatother
\newtheorem{ex}{Example}[section]
\makeatletter\let\c@ex\c@thm\makeatother

\makeatletter\let\c@observation\c@thm\makeatother

\makeatletter\let\c@warning\c@thm\makeatother
\newtheorem{digression}{Digression}[section]
\makeatletter\let\c@digression\c@thm\makeatother
\newtheorem{answ}{Answer}

\makeatletter
\let\c@equation\c@thm
\numberwithin{equation}{section}
\makeatother

\usepackage{hyperref}

\usepackage[capitalise]{cleveref}

\newcommand{\newrefformat}[2]{}

\crefname{lem}{Lemma}{Lemmas}
\crefname{thm}{Theorem}{Theorems}
\crefname{defn}{Definition}{Definitions}
\crefname{notn}{Notation}{Notations}
\crefname{const}{Construction}{Constructions}
\crefname{prop}{Proposition}{Propositions}
\crefname{rmk}{Remark}{Remarks}
\crefname{cor}{Corollary}{Corollaries}
\crefname{equation}{Display}{Displays}
\crefname{ex}{Example}{Examples}
\crefname{answ}{Answer}{Answers}

\newcommand{\cC}{\mathcal{C}}
\newcommand{\cD}{\mathcal{D}}

\newcommand{\cS}{\mathcal{S}}

\newcommand{\cat}{\cC\!\mathit{at}}
\newcommand{\set}{\cS\!\mathit{et}}

\newcommand{\sset}{\mathit{s}\set}
\newcommand{\strat}{\cS\!\mathit{trat}}

\newcommand{\psh}[1]{\set^{#1^{\op}}}
\newcommand{\spsh}[1]{\sset^{#1^{\op}}}
\newcommand{\ihom}[2]{#2^{#1}}
\newcommand{\pstrat}{\psh{t\Delta}}

\DeclareMathOperator{\op}{op}

\newcommand{\thin}[2]{t#1_{#2}}
\newcommand{\maybethin}[2]{(t)#1_{#2}}

\newcommand{\eqDelta}{\Delta[3]_{eq}}
\DeclareMathOperator{\colim}{colim}
\DeclareMathOperator{\id}{id}
\DeclareMathOperator{\im}{im}

\DeclareMathOperator{\Ob}{Ob}
\newcommand{\Fun}{\mathbf{Fun}} 
\DeclareMathOperator{\Ho}{Ho}

\DeclareMathOperator{\Hom}{Hom}
\DeclareMathOperator{\Map}{Map}

\newcommand{\aamalg}[1]{\underset{#1}{{\amalg}}} 
\DeclareMathOperator{\degr}{deg}
\addtotheorempostheadhook[rmk]{\crefalias{remarklisti}{rmk}}

\author{Viktoriya Ozornova}
\address{Fakult\"at f\"ur Mathematik, Ruhr-Universit\"at Bochum, D-44780 Bochum, Germany}
\email{viktoriya.ozornova@rub.de}

\author{Martina Rovelli}
\address{Department of Mathematics,
Johns Hopkins University,
Baltimore, MD 21218, USA}
\email{mrovelli@math.jhu.edu}

\begin{document}

\title{Model structures for $(\infty,n)$-categories on (pre)stratified simplicial sets and prestratified simplicial spaces}

\maketitle

\begin{abstract}
\vspace{-0.5cm}
We prove the existence of a model structure on the category of stratified simplicial sets whose fibrant objects are precisely $n$–complicial sets, which are a proposed model for $(\infty,n)$-categories, based on previous work of Verity and Riehl. We then construct a Quillen equivalent model based on simplicial presheaves over a category that can facilitate the comparison with other established models.
\end{abstract}

\tableofcontents

\section*{Introduction}

Grothendieck
and Quillen observed that the nerve of ordinary $1$-categories, taking values in simplicial sets, allows us to study the homotopy theory of categories. The nerve construction is a fully faithful embedding, and its essential image can be characterized as those simplicial sets which admit unique inner horn extensions.
Dropping the uniqueness requirements led to the notion of a ``quasi-category'', which is a model of an $(\infty,1)$-category. Joyal and Lurie took the first steps towards understanding the theory of $(\infty,1)$-categories by means of a model structure on the category of simplicial sets whose fibrant objects are precisely quasi-categories.

One might want to try and apply the same ideas to $2$-categories or more generally $n$-categories. When $n>1$, the approach has the caveat that the simplices in dimension at least $2$ have to play a double role, encoding both higher (possibly non-invertible) cells of an $n$-category, but also recording the composition of the lower-dimensional cells. As a manifestation of this issue, the nerve construction for $n$-categories, as defined by Street, fails to be fully faithful.

In order to solve this problem, Roberts
introduced the additional structure of a ``stratification'' on the nerve, declaring that the simplices corresponding to identity cells are ``marked''. Verity showed that the resulting Roberts--Street nerve, taking values in ``stratified simplicial sets'',
is a fully faithful functor whose image can be characterized as those stratified simplicial sets that admit unique extensions with respect to three classes of maps, recovering a similar picture to the case $n=1$.
These stratified simplicial sets go under the name of ``strict $n$-trivial complicial sets''.

Given that the naturally occurring examples of higher categories are not strict, one might want to modify the stratified nerve construction to accommodate the new framework.
To this end, rather than working with a stratification in which only identities are marked, one can instead consider the ``saturated'' stratification for the nerve of an $n$-category, in which the marked simplices are precisely the equivalences.
The resulting nerve is an ``$n$-trivial (weak) complicial set'', in the sense that it admits extensions with respect to the same classes of maps, but the lifts are no longer unique.
However, when endowed with the saturated stratification, the nerve gains the right lifting property with respect to a fourth kind of map,
as a manifestation of the fact that all equivalences are marked.

The $n$-trivial complicial sets that are also saturated, or for short~$n$-com\-plicial sets, are a proposed model of $(\infty,n)$-categories. 
This perspective, firstly proposed by Verity in unpublished work \cite{VeritySlides}, was widely explored by Riehl \cite[\textsection 3]{EmilyNotes}, and the existence of a model structure for $n$–complicial sets was conjectured.
It is expected that this model structure
on $\strat$ should be equivalent to the other established models of $(\infty,n)$-categories (cf~\cite[Conjecture 15.13]{BarwickSchommerPries}).
A variety of other models of $(\infty,n)$-categories are already known to be equivalent, see eg\ \cite{Ara,BarwickSchommerPries, br1, br2,haugseng}; amongst them we mention Rezk's $\Theta_n$-spaces \cite{rezkTheta}, which are fibrant objects in the category $\spsh{\Theta_n}$ of simplicial presheaves over the category $\Theta_n$.
The aim of this work is to establish the model structure for $n$-complicial sets and take the first steps towards the further comparison with $\Theta_n$-spaces.

Given that the models of $n$-complicial sets and of $\Theta_n$-spaces offer different advantages, an explicit comparison would yield a useful tool to export the constructions from one model to another. For instance, when working with (pre)stratified simplicial sets the pseudo-Gray tensor product is easy to define, as just the product of the underlying simplicial sets endowed with a certain stratification.

On the other hand, the globular approach of $\Theta_n$-spaces is a powerful setup to talk about dualities of an $(\infty,n)$-category.

In \cite{EmilyNotes}, Riehl conjectured (based on ideas by Verity) the existence of a model structure for $n$–complicial sets, as a special instance of a theorem of Verity  \cite{VerityComplicialI}, which gives conditions to obtain cartesian
model structures on the category $\strat$ of stratified simplicial sets.
In this paper we start by providing the verifications of these conditions, obtaining a proof of the desired model structure for $n$-complicial sets, which appears as \cref{modelstructureonstrat}.

\begin{unnumberedtheoremA}
For $n\ge0$, 
there is a cartesian model structure on the category $\strat$ of stratified simplicial sets, whose fibrant objects are precisely $n$-complicial sets.
\end{unnumberedtheoremA}

To record the piece of information given by a stratification on a simplicial set, one can identify a category $t\Delta$ obtained by adding to $\Delta$ new objects $[m]_t$ as well as structure maps $[m]\to[m]_t$ for $m\ge1$.
A ``stratified simplicial set'' is then a presheaf $X\colon t\Delta^{\op}\to\set$ with the further condition that the set of marked $k$-simplices $X([m]_t)$ is contained in the set $X([m])$ of ordinary $k$-simplices, namely that the new structure map $X([m]_t)\to X([m])$ is an inclusion.

With the further goal in mind of finding an explicit comparison of models between Rezk's $\Theta_n$-spaces and Riehl--Verity's $n$-complicial sets, we produce two intermediate model structures on the categories $\psh{t\Delta}$ and
$\spsh{t\Delta}$ of ``prestratified simplicial sets'' and ``prestratified simplicial spaces'', respectively.
In this paper, we construct such model structures and show that there are Quillen equivalences
$$\strat\rightleftarrows\psh{t\Delta}\rightleftarrows\spsh{t\Delta}.$$
Producing a Quillen equivalence
$$\spsh{t\Delta}\rightleftarrows\spsh{\Theta_n}$$
with the model of $\Theta_n$-spaces is the subject of an ongoing project, joint with Bergner.

The model structure that we put on $\psh{t\Delta}$ relies on Cisinski's theory of model categories of presheaves \cite{cisinski}, generalizing the techniques involved in establishing the Joyal model structure on $\sset$ for quasi-categories \cite{joyalnotes}.

\begin{unnumberedtheoremB}
For $n\ge0$, there is a model structure on the category $\pstrat$ of prestratified simplicial sets, whose fibrant objects are called $n$-precomplicial sets, and this model structure is Quillen equivalent to the model structure on the category $\strat$ for $n$-complicial sets.
\end{unnumberedtheoremB}

The construction of the model structure appears as \cref{modelstructureondiscretepresheaves}, and the Quillen equivalence as \cref{Quillenequivalence}.

Next, in order to create a Quillen equivalent model structure on simplicial presheaves over $t\Delta$ we specialize Ara's method from \cite{Ara}, which is in turn a generalization of methods of Joyal--Tierney \cite{JT} and Cisinski--Moerdijk \cite{CisinskiMoerdijk}.

\begin{unnumberedtheoremC}
For $n\ge0$, there is a cartesian model structure on the category $\spsh{t\Delta}$ of prestratified simplicial spaces, whose fibrant objects are called $n$-precomp\-licial spaces, and this model category is Quillen equivalent to the model structure on $\pstrat$ for $n$-complicial sets.
\end{unnumberedtheoremC}

The construction of the model structure appears as 
\cref{interestingmodelstructure}, and  the Quillen equivalence as \cref{secondQuillenequivalence}.

The model of $n$-precomplicial spaces, which is morally a generalization
of Rezk's model for complete Segal spaces \cite{rezkhomotopy}, is also interesting in itself. On the one hand, it carries all the advantages of model categories of simplicial presheaves; for instance, it is easy to define $n$-precomplicial objects in a different model category. On the other hand, the indexing category $t\Delta$ does not depend on $n$, so describing an $n$-precomplicial space for large $n$ requires the same amount of data as for small values of $n$.

In an ongoing project, we aim to achieve the final comparison with $\Theta_n$-spaces.

\begin{unnumberedconjecture}
The model structure on the category $\spsh{t\Delta}$ of prestratified simplicial spaces for $n$-precomplicial spaces is Quillen equivalent to the model structure on Rezk's model structure on $\spsh{\Theta_n}$ for $\Theta_n$-spaces.
\end{unnumberedconjecture}

The equivalence of models has already been established for $n=1$, as a consequence of the fact that Lurie’s model structure for naturally marked simplicial sets and Rezk’s model structure for complete Segal spaces are Quillen equivalent, and for $n=2$, as a consequence of work by Lurie \cite{LurieGoodwillie} and by Gagna, Harpaz and Lanari \cite{GHL}.

 \addtocontents{toc}{\protect\setcounter{tocdepth}{1}}

\subsection*{The case $n=\infty$}
As a final note, we point out that this paper focuses on $n$-(pre)complicial sets as a model of $(\infty,n)$-categories for finite $n$, but the constructions and theorems are also valid for the case $n=\infty$. In particular, one gets a notion of an \emph{$\infty$-(pre)complicial set}, as well as a model structure for $\infty$-(pre)complicial sets, and might wish to relate them to some version of an $(\infty,\infty)$-category.

A precise notion of an $(\infty,\infty)$-category has not been formally established. However, based on the notion of an $(\infty,n)$-category for all finite $n$, some of the experts seem to agree that there are (at least) two meaningful approaches one can take. In a mathoverflow post \cite{jf}, Barwick, Rezk and Schommer-Pries consider two different notions for the $(\infty,1)$-category of $(\infty,\infty)$-categories, each given by the homotopy limit of a tower of functors from the $(\infty,1)$-category of $(\infty,n+1)$-categories to the $(\infty,1)$-categories of $(\infty,n)$-categories. The functors occurring in these towers are respectively the right and left adjoints to the inclusions of $(\infty,n)$-categories into $(\infty,n+1)$-categories.

Roughly speaking, in the first case, referred to as \emph{inductive}, a map of $(\infty,\infty)$-categories is a weak equivalence if and only if it induces equivalences of \emph{underlying} $(\infty,n)$-categories, and in the second case, referred to as \emph{coinductive}, if and only if it induces equivalences of \emph{truncated} $(\infty,n)$-categories. The different nature of the two notions is already present when working with strict higher categories; for instance, Lafont--M\'etayer--Worytkiewicz \cite{LMW} define $\omega$-categories following the coinductive point of view.

We believe that $\infty$-(pre)complicial sets should provide a model for the inductive notion of $(\infty,\infty)$-categories. Moreover, they can conjecturally be used to implement the coinductive notion, by localizing with respect to the nerve of the free $\infty$-categorical adjoint equivalence (denoted by $\textbf{P}^1$ in \cite{LMW}), considered as maximally marked. These topics will be addressed in future work.

 \addtocontents{toc}{\protect\setcounter{tocdepth}{1}}
 
\subsection*{Acknowledgements}
The authors are grateful to Emily Riehl and Dominic Verity for generously sharing some of their drafts and ideas, and for giving us valuable inputs. We also would like to thank Julie Bergner, Magdalena K\c edziorek and Lennart Meier 
for useful conversations. The exposition of the paper benefited greatly from the referee's comments. 
The second-named author was partially funded by the Swiss National Science Foundation, grant P2ELP2\textunderscore172086.  The authors would like to thank the Isaac Newton Institute for Mathematical Sciences, Cambridge, for support and hospitality during the program ``Homotopy Harnessing Higher Structures'' where work on this paper was undertaken. This work was supported by EPSRC grant numbers EP/K032208/1 and EP/R014604/1.

\addtocontents{toc}{\protect\setcounter{tocdepth}{2}}

\section{The model structure for $n$-(pre)complicial sets}
\label{section1}
In this section, we define the categories $\strat$ and $\pstrat$ of stratified and prestratified simplicial sets, and put Quillen equivalent model structures on them. The fibrant objects, called ``$n$-(pre)complicial sets'', are a proposed model for $(\infty,n)$-categories.

\subsection{Prestratified and stratified simplicial sets}

We start by giving a description of Verity's category $t\Delta$ from \cite{VerityComplicialI}, in terms of its generators and relations.

\begin{notn}
\label{tDelta}
Let $t\Delta$ be the category defined as follows.
The set of objects is given by
$$\Ob(t\Delta):=\{[m]\text{ for all }m\geq 0\}\cup\{[m]_t\text{ for all }m\geq 1\}.$$
The maps in $t\Delta$ are generated under composition by the following four kinds of maps:
\begin{itemize}
\item cofaces $d^i\colon[m]\to[m+1]$ for $0\leq i \leq m+1$,
\item codegeneracies $s^i\colon[m]\to[m-1]$ for $0 \leq i \leq m-1$,
\item counmarking maps $\varphi\colon[m]\to[m]_t$ for $m\geq 1$,
\item comarking maps 
$\zeta^{i}\colon[m]_t\to[m-1]$ for $0\leq i \leq m-1$. 
\end{itemize}
subject to the usual cosimplicial identities
   \begin{alignat*}{2}
    d^jd^{i}  =  d^id^{j-1}\colon &[m-2]\to [m]  \mbox{ if } 0 \leq i<j\leq m\\
    s^jd^{i}  =  d^is^{j-1} \colon& [m] \to [m]\mbox{ if } 0\leq i<j\leq m\\
    s^jd^j  =  \id =s^jd^{j+1} \colon& [m] \to [m]\mbox{ if }0 \leq j \leq m\\
    s^jd^i  =  d^{i-1}s^j \colon& [m] \to [m]\mbox{ if } 1\leq j+1 < i \leq m+1\\ 
    s^js^i  =  s^is^{j+1}\colon& [m+2] \to [m]\text{ if }0\leq i \leq j \leq m
   \end{alignat*}
and the additional relations
\begin{alignat*}{2}
\zeta^i\varphi &=s^i\colon [m+1]\to[m]\text{ for }m\ge1\text{ and }0\leq i \leq m,\\
s^i\zeta^{j+1} &=s^j\zeta^{i} \colon [m+2]_t \to [m]\text{ for }0\leq i \leq j \leq m.
\end{alignat*}
The generating morphisms of $t\Delta$ can be pictured as follows.
\[
\begin{tikzcd}[row sep=large, column sep=large]
{[0]}
 \arrow[r, arrow, shift left=1ex] \arrow[r, arrow, shift right=1ex] 
& 
{[1]}\arrow[l, arrow]  \arrow[r, arrow] \arrow[r, arrow, shift left=1.5ex] \arrow[r, arrow, shift right=1.5ex]
\arrow[d, arrow, "\varphi"]& 
{[2]}
\arrow[l, arrowshorter, shift left=0.75ex] \arrow[l, arrowshorter, shift right=0.75ex] 
\arrow[d, arrow, "\varphi"]&[-0.8cm]\cdots\\
 & 
{[1]_t}  
\arrow[ul, arrowshorter, "\zeta^0"]
& 
{[2]_t} \arrow[ul, arrowshorter, shift left=0.5ex, "\zeta^0"] \arrow[ul, arrowshorter, shift right=0.5ex,"\zeta^1" swap] 
&\cdots
\end{tikzcd}
\]
\end{notn}

\begin{digression}
Verity essentially gives this description of $t\Delta$ in terms of generators and relations in \cite[Observation 12]{VerityComplicialI}. The list of relations is however not exhaustive, as the family of relations $s^k\circ\zeta^{l+1} =s^l\circ\zeta^{k}$ is omitted\footnote{Without the relation $s^k\circ\zeta^{l+1} =s^l\circ\zeta^{k}$, we almost recover instead Street's category $h\Delta$ from \cite[Section 5]{StreetOrientedSimplexes}, modulo the fact that $h\Delta$ does not contain $[1]_t$ amongst its objects.}.
The complete description of $t\Delta$ will appear in \cite{RiehlVerityBook}.
\end{digression}

We are interested in presheaves over $t\Delta$.

\begin{defn}
A \emph{prestratified simplicial set} is a presheaf $X\colon t\Delta^{\op}\to\set$. We write
$$X_m:=X([m])\text{ and }\thin{X}{m}:=X([m]_t),$$
and we write
$$\maybethin{X}{m}:=X([m]_{(t)})$$
for the value of $X$ at a generic element $[m]_{(t)}$ of $t\Delta$.
We denote by $\psh{t\Delta}$ the category of prestratified simplicial sets.
\end{defn}

The inclusion $\Delta\hookrightarrow t\Delta$ induces a functor
$$U\colon\pstrat\to\sset$$
which sends a prestratified simplicial set to its underlying simplicial set. In particular, to describe a (pre)stratified simplicial set it is enough to give the underlying simplicial set, and describe which simplices are marked (possibly with multiple labels). The first of the two extra relations present in the indexing category $t\Delta$ translates to saying that any degenerate simplex is marked, with a distinguished label in the case of multiple marking. The second extra relation says then that these distinguished markings coincide whenever the simplex can be expressed as an iterated degeneracy in different ways.

\begin{rmk}
\label{sharpandflat}
The forgetful functor $\pstrat \to \sset$ mapping a prestratified simplicial set to its underlying simplicial set has both a left and a right adjoint,
$$(-)^\flat\colon \sset \rightleftarrows \pstrat\colon U\text{ and }U\colon \pstrat \rightleftarrows \sset\colon (-)^\sharp.$$
The left adjoint 
$(-)^\flat$ assigns to a simplicial set $X$ its minimal stratification $X^\flat$, where only degenerate simplices are marked. We regard this one as the canonical stratification on a simplicial set, and often omit the notation.
Similarly, the right adjoint $(-)^\sharp$ assigns to a simplicial set $X$ its maximal stratification $X^\sharp$ where all simplices marked in positive degrees.
\end{rmk}

We record for further reference the notation for representable prestratified simplicial sets.

\begin{notn}
\label{notationrepresentable}
We denote by
\begin{itemize}
    \item $\Delta[m]$ the prestratified simplicial set represented by $[m]$ for $m\ge0$. In particular, only degenerate simplices are marked. We observe that $\Delta[0]$ is the terminal object of $\pstrat$.
    \item $\Delta[m]_t$ the prestratified simplicial set represented by $[m]_t$ for $m\ge1$. In particular, besides the degenerate simplices, the only non-degenerate $m$-simplex is marked. 
    \item $\Delta[-1]$ the prestratified simplicial set that is constant at the empty set, ie, given componentwise by $\maybethin{\Delta[-1]}{k}=\varnothing$.
    We observe that $\Delta[-1]$ is the initial object of $\pstrat$.
\end{itemize}
\end{notn}

A direct verification (or an instance of \cite[Proposition 1.2.27]{cisinski}) shows the following characterization of monomorphisms in $\pstrat$.

\begin{lem}
\label{generatingmonos}
The class of monomorphisms in $\pstrat$ is the saturation in the sense of \cref{saturation}
of the set of maps
 $$I:=\{\partial\Delta[m]\hookrightarrow\Delta[m]\ |\  m \geq 0 \}\cup\{\Delta[m]\hookrightarrow\Delta[m]_t\ |\ m \geq 1\}.$$
\end{lem}

Any monomorphism of prestratified simplicial sets factors into two pieces: a ``regular'' inclusion and an ``entire'' inclusion, which we now define. These notions agree with Verity's original ones from \cite[Definition 9]{VerityComplicialI} in the case of maps of ``stratified simplicial sets'', which will be defined later.

\begin{defn}
A monomorphism of prestratified simplicial sets $j \colon X \hookrightarrow Y$ is
\begin{enumerate}
    \item \emph{entire} if it is an identity on the underlying simplicial set.
    \item \emph{regular} if
    for every $m\geq 1$, the following diagram is a pullback. 
\[
\begin{tikzcd}
\thin{X}{m} \arrow[r, "\varphi"] \arrow[d, "j" swap, hook] & X_m\arrow[d, "j", hook]\\
\thin{Y}{m} \arrow[r, "\varphi" swap]  & Y_m
\end{tikzcd}
\]
    \end{enumerate}
\end{defn}

\begin{rmk}
The class of regular inclusions is closed under compositions, retracts,  transfinite compositions, and products with any prestratified simplicial set.
\end{rmk}

Verity works with prestratified simplicial sets that satisfy a further condition.

\begin{defn}[{\cite[Definition 5]{VerityComplicialI}}]
A \emph{stratified simplicial set}\footnote{We warn the reader
that Verity's use of the terminology ``stratified'' is unrelated to that of Ayala--Francis--Tanaka  \cite{AyalaFrancisTanaka} and Lurie \cite{LurieHA}.}
is a prestratified simplicial set $X\colon t\Delta^{\op}\to\set$ such that 
the structure maps
$$\thin{X}{m}\to X_m$$
induced by the coforgetful maps $\varphi\colon[m]\to[m]_t$ 
are injective for all $m\ge 1$.
We denote by $\strat$ the full subcategory of $\pstrat$ given by the stratified simplicial sets.
\end{defn}

For a stratified simplicial set $X\colon t\Delta^{\op}\to\sset$, we think of $X_m$ as the set of $m$-simplices, and of $\thin{X}{m}\subset X_m$ as the subset of \emph{marked} or \emph{thin} $n$-simplices.
Following the same intuition, when $X$ is only prestratified an $m$-simplex can be marked admitting \emph{multiple labels}, and $\thin{X}{m}$ is then a set of labels of marked $m$-simplices. An example of a prestratified simplicial set that is not stratified will be given in \cref{counterexamplepushout}.

\begin{rmk}
\label{DeltaintDelta}
All the representable prestratified simplicial sets from \cref{notationrepresentable} are stratified simplicial sets and the full subcategory of $\strat$ given by the standard simplices $\Delta[m]$ for $m\ge0$ and the standard marked simplices $\Delta[m]_t$ for $m\ge1$ is isomorphic to $t\Delta$. In particular, in this category we see that, beside the usual cosimplicial maps between standard simplices there are entire maps $\varphi\colon\Delta[m]\to\Delta[m]_t$ for $m\ge0$, and for $m\ge1$ the $i$-th codegeneracy induces a map $\zeta^i\colon\Delta[m]_t\to\Delta[m-1]$.
\end{rmk}

\label{categoricalproperties}
As pointed out in \cite[Observation 12]{VerityComplicialI}, $\strat$ is a reflective subcategory of $\pstrat$, as the inclusion $i$ fits into an adjoint pair
$$R\colon\pstrat\rightleftarrows\strat\colon i.$$
The left adjoint, which we call the \emph{reflector}, can be computed levelwise on objects as 
$$(RX)_n:=X_n\text{ and }\thin{RX}{n}:=\im(\thin{X}{n}\to X_n).$$
In particular, limits and colimits in $\pstrat$ are pointwise, limits in $\strat$ are computed in $\pstrat$, and colimits in $\strat$ are obtained by applying the reflector to the colimits in $\pstrat$.

\begin{ex}
\label{counterexamplepushout}
Note that pushouts in $\pstrat$ and in $\strat$ might differ a lot. Indeed, given a diagram of stratified simplicial sets and monomorphisms, the pushout in $\pstrat$ might not be a stratified simplicial set, and therefore cannot be the pushout in $\strat$. For example, the pushout in $\pstrat$ of the diagram of entire inclusions
\[
\begin{tikzcd}
\Delta[m]_t &\Delta[m] \arrow[l, hook'] \arrow[r, hook]& \Delta[m]_t
\end{tikzcd}
\]
is not a stratified simplicial set.
\end{ex}

However, as an instance of the following proposition, sometimes the push\-outs in $\strat$ and in $\pstrat$ do coincide.

\begin{prop}
\label{RegularPushout}  The pushout in $\pstrat$ of a regular inclusion of stratified simplicial sets along any map of stratified simplicial sets is a regular inclusion of stratified simplicial sets.
\end{prop}

One can prove this proposition by means of the following lemma.

\begin{lem}[Reedy's Lemma]
\label{ReedyLemmaSets}
Consider a diagram of sets as follows.
\[
\begin{tikzcd}[row sep=0.3cm]
A \arrow[rr, "f"] \arrow[dr,swap,"\alpha"] \arrow[dd,swap,"j"] &&
  B \arrow[dd] \arrow[dr,"\beta"] \\
& A' \arrow[rr,crossing over, "f'" near start] &&
  B' \arrow[dd] \\
C \arrow[rr] \arrow[dr,swap,"\gamma"] && D \arrow[dr,swap,"\delta"] \\
& C' \arrow[rr] \arrow[uu,<-,crossing over, "j'" near end]&& D'
\end{tikzcd}
\]
Suppose that the front and the back faces are pushouts, the maps $\alpha, \beta, \gamma, j, j'$ are injective, and the left-hand face is a pullback.
Then:
\begin{enumerate}
\item the map $\delta$ is injective, and
\item the right-hand face is also a pullback.
\end{enumerate}
\end{lem}

\begin{proof}[Proof of \cref{RegularPushout}]
Given a span 
\[
\begin{tikzcd}
Y &X \arrow[l, hook', "j" swap] \arrow[r, "f"]& Z,
\end{tikzcd}
\]
of stratified simplicial sets such that $j$ is a regular inclusion, we apply \cref{ReedyLemmaSets} to the following diagram
\[
\begin{tikzcd}[row sep=0.3cm]
\thin{X}{m} \arrow[rr, "f"] \arrow[dr,swap,"\varphi"] \arrow[dd,swap,"j"] &&
  \thin{Z}{m} \arrow[dd] \arrow[dr,"\varphi"] \\
& X_m \arrow[rr,crossing over, "f" near start] &&
  Z_m \arrow[dd] \\
\thin{Y}{m}\arrow[rr] \arrow[dr,swap,"\varphi"] && \thin{P}{m} \arrow[dr,swap,"\varphi"] \\
& Y_m \arrow[rr] \arrow[uu,<-,crossing over, "j" near end]&& P_m,
\end{tikzcd}
\]
where $P$ denotes the pushout of the span above in $\pstrat$. 
We conclude that the pushout in $\pstrat$ is a stratified simplicial set by \cref{ReedyLemmaSets}(1), and that the pushout of $j$ is again a regular inclusion by \cref{ReedyLemmaSets}(2).
\end{proof}

\begin{rmk} 
We record two closure properties of $\strat$.
\label{rmkretracts}
\begin{enumerate}
    \item A retract in $\pstrat$ of an object in $\strat$ is again in $\strat$.
    \item The filtered colimit in $\pstrat$ of objects in $\strat$ is again in $\strat$. In particular, filtered colimits in $\strat$ and $\pstrat$ coincide.
\end{enumerate}
\end{rmk}

\begin{prop}
\label{cartesianclosed}
The categories $\pstrat$ and $\strat$ are cartesian closed, with internal homs given by the prestratified simplicial set
$$\maybethin{(Z^Y)}{m}:=\Hom_{\pstrat}(\Delta[m]_{(t)}\times Y,Z).$$
\end{prop}

When it is useful to emphasize the ambient category, we write $\hom_{\strat}(Y,Z)$ or $\hom_{\pstrat}(Y,Z)$ for the internal hom $Z^Y$.

\begin{proof}
The fact that the above formula for $Z^Y$ yields the internal hom for the category of presheaves $\psh{t\Delta}$ is formal and a special case of \cref{presheafcartesianclosed}. Moreover, we will show that if $Z$ is a stratified simplicial set and $Y$ any prestratified simplicial set, then the internal hom $Z^Y$ of $\psh{t\Delta}$ is also a stratified simplicial set. Using this fact, it is easy to see that the same formula yields the internal hom for the category $\strat$ whenever $Y$ and $Z$ are stratified simplicial sets. The fact that the category $\strat$ is cartesian closed
is also mentioned in \cite[Observation 12 and Definition 59]{VerityComplicialI}.

To justify the claim, we need to check that for any stratified simplicial set $Z$ and for any $t\Delta$-set $Y$
the map
$\thin{(Z^Y)}{k}\to (Z^Y)_k$
induced by $\varphi\colon [k] \to [k]_t$ for $k\ge1$ is injective.
This map is precisely the map
\[
\Hom_{\pstrat}(\Delta[k]_{t}\times Y,Z) \to \Hom_{\pstrat}(\Delta[k]\times Y,Z)
\]
induced by precomposition with $\varphi \times \id\colon \Delta[k]\times Y\to \Delta[k]_{t}\times Y$. So we want to show that a lift of a map of presheaves 
$f\colon \Delta[k]\times Y \to Z$ to a map $\widetilde{f}\colon \Delta[k]_{t}\times Y \to Z$ is unique if it exists. Since the underlying simplicial sets of $\Delta[k]\times Y$ and $\Delta[k]_{t}\times Y$ are equal, we only need to check that the value of $\widetilde f$ on marked simplices is uniquely determined by $f$.
Moreover, since the marked simplices of $\Delta[k]\times Y$ and $\Delta[k]_{t}\times Y$ coincide in all dimensions except $k$, we only need to check that the value of $\widetilde f$ on $[k]_t$, namely the function
\[\widetilde{f}_{[k]_t}\colon(\Delta[k]_t\times Y)([k]_{t})\to Z([k]_{t}),\]
is uniquely determined by $f$. Since $Z$ is by assumption a stratified simplicial set, the structure map $\varphi_k\colon Z([k]_t)\to Z([k])$ is injective, so it is enough to check that the function
\[\varphi_k\circ\widetilde{f}_{[k]_t}\colon(\Delta[k]_t\times Y)([k]_{t})\to Z([k]_{t})\to Z([k]),\]
is uniquely determined by the $f$.
We conclude recalling that $\widetilde f$ is uniquely determined on the underlying simplicial set and observing that there is a commutative diagram of the following form
\[\begin{tikzcd}
(\Delta[k]_{t}\times Y)([k]_{t})\arrow[r, "{\widetilde{f}_{[k]_t}}"]  \arrow[d, "\varphi_k"swap]& Z([k]_{t})\arrow[d, hook, "\varphi_k"]\\
(\Delta[k]_{t}\times Y)([k]) \arrow[r, "{\widetilde{f}_k=f_k}"swap]& Z([k]).
\end{tikzcd}\]
\end{proof}

\begin{rmk}
\label{inclusionfullyfaithful}
The inclusion $\strat\hookrightarrow\psh{t\Delta}$ is fully faithful as a functor enriched over $\pstrat$, ie, it induces an isomorphism of prestratified simplicial sets
$$i\colon\hom_{\strat}(X,V)\cong\hom_{\pstrat}(X,V).$$
\end{rmk}

\begin{rmk}
\label{enrichedadjunction}
The reflector is a functor enriched over $\pstrat$, ie, it induces a map of prestratified simplicial sets 
$$R\colon\hom_{\pstrat}(Y,Z)\to i\hom_{\strat}(RY,RZ).$$
Moreover, the adjunction is enriched over $\pstrat$, ie, there are natural isomorphisms of prestratified simplicial sets 
$$i\hom_{\strat}(RY,X)\cong\hom_{\pstrat}(Y,X)$$
which are an easy consequence of the fact that $R$ commutes with products. 
\end{rmk}

\subsection{Precomplicial and complicial sets}

In this subsection, we describe which (pre)stratified simplicial sets should be thought of as $(\infty,n)$-categories, which in \cite{RiehlVerityBook} go under the name of ``$n$-(pre)complicial sets''. The condition that we require is having the right lifting property with respect to four classes of elementary anodyne maps.
For an account on the purpose and the intuition behind each of them, see the expository note \cite{EmilyNotes}.

The elementary anodyne maps involve the following stratified simplicial sets, which can also be found in \cite{EmilyNotes}, and partially in
\cite[Notation 10]{VerityComplicialI}\footnote{For historical reasons, Street and Verity first focused on stratified simplicial sets that had the right lifting properties with respect to the maps of type (1) and (2), and refer to them as ``weak $\omega$-categories'' \cite{StreetOrientedSimplexes} and ``weak complicial sets'' \cite{VerityComplicialI}, respectively. A weak complicial set that also has the right lifting property with respect to the maps of type (3) is called ``$n$-trivial''. Later on, Riehl introduced in \cite{EmilyNotes} the terminology ``saturated'' for $n$-trivial weak complicial sets that also have the right lifting property with respect to maps of type (4).}.

\begin{notn} 
\label{preliminarynotation}
We denote by 
\begin{enumerate}
    \item $\Delta^k[m]$, for $0\leq k \leq m$, the stratified simplicial set whose underlying simplicial set is $\Delta[m]$ and in which a non-degenerate simplex is marked if and only if it contains the vertices $\{k-1,k,k+1\}\cap [m]$.
    \item $\Delta^k[m]'$, for $0\leq k \leq m$, the stratified simplicial set obtained from $\Delta^k[m]$ by additionally marking the $(k-1)$-st and $(k+1)$-st face of $\Delta[m]$,
    \item $\Delta^k[m]''$, for $0\leq k \leq m$, the stratified simplicial set obtained from $\Delta^k[m]'$ by additionally marking the $k$-th face of $\Delta[m]$.
    \item $\Lambda^k[m]$, for $0\leq k \leq m$, the stratified simplicial set whose underlying simplicial set is the $k$-horn $\Lambda^k[m]$ and whose non-degenerate marked simplices consist of those simplices that are marked in $\Delta^k[m]$.
    \item $\eqDelta$ the stratified simplicial set whose underlying simplicial set is $\Delta[3]$, and whose non-degenerate marked simplices consist of all $2$- and $3$-simplices, as well as $1$-simplices $[02]$ and $[13]$.
\end{enumerate}
\end{notn}

We also make use of the join construction for stratified simplicial sets $\star\colon\strat\times\strat\to\strat$. The underlying simplicial set of this join construction agrees with Joyal's join construction of the underlying simplicial sets from \cite[\textsection6.6]{joyalnotes} or \cite[App.D.2]{RiehlVerityBook}, ie, given simplicial sets $A$ and $B$ the set of $k$-simplices of $A\star B$ is given by
\[(A\star B)_k=\coprod_{l=-1}^{k}(A_{l}\times B_{k-l-1}),\] 
following the convention that $A_{-1}=B_{-1}=*$.
 The stratification of the join construction can then be found in \cite[Observation 34]{VerityComplicialI} or 
 \cite[App.D.2]{RiehlVerityBook}. Explicitly, a $k$-simplex in the join $A\star B$ can be seen as a map $\alpha\star\beta\colon\Delta[l]\star \Delta[k-l-1] \to A\star B$ with $\alpha$ an $l$-simplex in $A$ and $\beta$ a $(k-l-1)$-simplex in $B$, with the convention $\Delta[-1]=\varnothing$; the simplex $\alpha\star\beta$ is then marked if and only if either $\alpha$ or $\beta$ is marked or both.
\begin{defn}
\label{anodynemaps}
Let $n\ge0$. An \emph{elementary anodyne extension} is one of the following.
 \begin{enumerate}
  \item  The \emph{complicial horn extension} from \cite[Definition 15]{VerityComplicialI}, ie, the regular inclusion
 $$\Lambda^k[m]\to \Delta^k[m]\text{ for $m\geq 1$ and $0\leq k\leq m$};$$
 \item The \emph{complicial thinness extension} from \cite[Definition 15]{VerityComplicialI}, ie, the entire inclusion 
$$\Delta^k[m]' \to \Delta^k[m]''\text{ for $m\geq 2$ and $0\leq k \leq m$};\footnote{We point out that the cases $k=0,m$ behave quite differently from the cases $0<k<m$. For instance, in the former there are marked $1$-simplices, namely $[0,1]$ or $[m-1,m]$, which are used to encode Joyal's special outer horn lifting property as explained in \cite[App. D.4]{RiehlVerityBook}.}$$
\item The \emph{$l$-th triviality extension} map,
ie, the entire inclusion
$$\Delta[l]\to \Delta[l]_t\text{ for $l>n$};$$
\item The \emph{saturation extension} from \cite[Definition 3.2.7]{EmilyNotes}, ie, the entire inclusion
$$\Delta[l]\star \Delta[3]_{eq}  \to \Delta[l]\star \Delta[3]^{\sharp}\text{ for $l\geq -1$}.$$
\end{enumerate}
We denote by $\Lambda_n$ the collection of all elementary anodyne extensions. 
\end{defn}

\begin{rmk}
An involved combinatorial argument shows that, in the presence of the maps of type (1),(2) and (3), requiring a right lifting property with respect to (4) is equivalent to requiring it with respect to the family of maps
$$\Delta[m]\star \Delta[3]_{eq} \star \Delta[l] \to \Delta[m]\star \Delta[3]^{\sharp}\star \Delta[l]\text{ for $l,m\geq -1$}$$
instead. We refer the reader to \cite[Appendix D]{RiehlVerityBook} for more details.
\end{rmk}

The following terminology is borrowed from \cite{RiehlVerityBook}.

\begin{defn}
For $n\ge0$, an \emph{$n$-complicial set} is a stratified simplicial set that has the right lifting property with respect to the elementary anodyne maps from \cref{anodynemaps}.
\end{defn}

Roughly speaking, a stratified simplicial set $W$ consists of a set of objects $W_0$, and for $k>0$ a set of $k$-morphisms $W_k$ and a set of $k$-equivalences $\thin{W}{k}$.
According to this intuition, lifting with respect to
\begin{enumerate}
    \item all complicial horn extensions guarantees that cells can be suitably composed;
    \item all complicial thinness extensions guarantees that any composite of equivalences is also one;
    \item $n$-triviality anodyne extensions guarantees that all cells in degree higher than $n$ are invertible;
    \item all saturation anodyne extensions guarantees that all equivalences are marked.
    
\end{enumerate}
In this sense, it is fair to regard an $n$-complicial set as an $(\infty,n)$-category.

\begin{rmk}
We point out that the parameter $n$ only appears in the triviality elementary anodyne extension. In particular, if $n\ge n'$ then $\Lambda_n\subset\Lambda_{n'}$ and any $n'$-complicial set (which is an $(\infty,n')$-category) is also an $n$-complicial set (which is an $(\infty,n)$-category).
Let us elaborate on what $n$-complicial sets recover for low values of $n$.
\begin{enumerate}
    \item[(0)] A $0$-complicial set is precisely a Kan complex endowed with the maximal stratification.
    \item A $1$-complicial set is a quasi-category in which the marked $1$-simplices are precisely the equivalences, and all simplices in degree higher than $1$ are marked; so $1$-complicial sets are essentially the same objects are Lurie's ``naturally marked quasi-categories'' from \cite{htt}.
    \item The underlying scaled simplicial set of a $2$–complicial set is a (weak) $\infty$–bicategory in the sense of \cite{LurieGoodwillie}, as shown by Gagna--Harpaz--Lanari in \cite{GHL}.
    \end{enumerate}
\end{rmk}

We also instead want to consider prestratified simplicial sets (which are not necessarily stratified) that satisfy the same lifting properties.

\begin{defn}
For $n\ge0$, an \emph{$n$-precomplicial set}\footnote{We warn the reader that Verity \cite{VerityComplicialAMS} uses the terminology ``pre-complicial'' to mean something different.
} is a prestratified simplicial set that has the right lifting property with respect to the elementary anodyne maps from \cref{anodynemaps}.
\end{defn}

In the next section we construct model structures for $n$-complicial and $n$-precomplicial sets.

\subsection{The model structure for $n$-complicial sets}

The model structure on $\strat$ for $n$-complicial sets was first claimed in \cite[Example 3.3.6]{EmilyNotes} as an instance of \cite[Theorem 100]{VerityComplicialI}. We provide the technical details of the proof.

The weak equivalences, which we call \emph{$\Lambda_n$-local equivalences}, are defined in terms of the set\footnote{We warn the reader that the index $n$ in $\Lambda_n$ does not refer to the $n$-fold iterated $\Lambda$-construction from \cite[\textsection 2.10]{Ara}.} $\Lambda_n$
of elementary anodyne maps, and in terms of the class of $\Delta[1]_t$-homotopy equivalences in $\strat$.

The notion of $\Delta[1]_t$-homotopy equivalence in $\pstrat$ is given in \cref{homotopyequivalence}, and agrees with the one given in \cite[Section 6.1]{VerityComplicialI} for maps that are in $\strat$.

\begin{defn}
A map of prestratified simplicial sets $f\colon X\to Y$ is a \emph{$\Lambda_n$-local equivalence} if and only if for any $n$-complicial set $Z$, the induced map on internal homs $f^*\colon \ihom{Y}{Z}\to \ihom{X}{Z}$ is a $\Delta[1]_t$-homotopy equivalence.
\end{defn}

\begin{thm}
\label{modelstructureonstrat}
For $n\ge0$, the category $\strat$ admits a cartesian, left proper, combinatorial 
model structure where
\begin{itemize}
\item the cofibrations are precisely the monomorphisms,
\item the fibrant objects are precisely the $n$-complicial sets, and
\item the weak equivalences are precisely the $\Lambda_n$-local weak equivalences.
\end{itemize}
We call this the \emph{model structure for $n$–complicial sets}.
\end{thm}

The key fact to check is the following proposition. We recall that a $\Lambda_n$-anodyne extension in $\strat$ is a map in $\strat$ that can be written as a retract of a transfinite composition of pushouts of elementary $\Lambda_n$-elementary anodyne extensions.

\begin{prop} \label{AnodyneTimesMonoStrat}
For $n\ge0$, the pushout-product of a $\Lambda_n$-anodyne extension $I\to J$ and a monomorphism $K\hookrightarrow L$,
    $$(I\times L)\aamalg{I\times K}(J\times K)\to J\times L,$$
    is an anodyne extension in $\strat$.
\end{prop}

\begin{proof}[Proof of \cref{AnodyneTimesMonoStrat}]
The proof of the proposition consists of the verifications of all possible cases of a generating monomorphism from \cref{generatingmonos}, with an elementary anodyne extension from \cref{anodynemaps}. 
When the elementary anodyne extension is a complicial horn or thinness anodyne extension, this was first done in \cite[Lemma 72 and Observation 74]{VerityComplicialI}, and a detailed argument is also now available in \cite[Proposition D.3.8]{RiehlVerityBook}, and partially in \cite{HenryWeakModel}. When the elementary anodyne extension is a triviality or saturation anodyne extension, this is proven in \cref{PPsaturationboundary,PPsaturationmarking,PPtrivialityboundary,PPtrivialitymarking}.
\end{proof}

We can now prove the theorem.

\begin{proof}[Proof of \cref{modelstructureonstrat}]
The existence of the desired cofibrantly generated model structure is an application of Verity's theorem \cite[Theorem 100]{VerityComplicialI} for model structures on $\strat$, which is in turn an instance of Smith's Theorem for the existence of model structures. 
In order to use Verity's theorem, we need to check that the set $\Lambda_n$ satisfies Conditions (i)-(ii) of \cite[Definition 91]{VerityComplicialI}.
\begin{enumerate}
\item[(i)\ ] First, we observe that the set $\Lambda_n$ contains by definition Verity's elementary anodyne maps, which are precisely the complicial horn inclusions and the complicial thinness extensions from \cref{anodynemaps}.
\item[(ii)\ ] Next, we show that for every map $ I \to J$ in $\Lambda_n$ and for every $n$-complicial set $X$, the induced map on internal homs
$X^J \to X^I$
is a $\Delta[1]_t$-homotopy equivalence. 
First, we claim that this map is an ``acyclic fibration'', ie, it has the right lifting property with respect to all monomorphisms. For this, we consider a lifting problem for a monomorphism $K\to L$:
\[
\begin{tikzcd}
K \arrow[r]\arrow[d, hook] & \ihom{J}{X} \arrow[d, "\alpha^*"]\\
L\arrow[r] \arrow[ur, dashed, "?"] & \ihom{I}{X}.
\end{tikzcd}
\]
By adjointness, the lifting problem is equivalent to the following one
\[
\begin{tikzcd}
K\times J \aamalg{K\times I} L\times I \arrow[r]\arrow[d, hook] & X. \\
L\times J  \arrow[ur, dashed, swap , "?"] &
\end{tikzcd}
\]
Given that the vertical arrow is a $\Lambda_n$-anodyne extension by \cref{AnodyneTimesMonoStrat} and  since $X$ is an $n$-complicial set, a lift exists. 
We conclude knowing that any acyclic fibration
between stratified simplicial sets is a $\Delta[1]_t$-homotopy
equivalence by \cite[Observation 90]{VerityComplicialI}.
\end{enumerate}
Given that every object is cofibrant, the model structure is left proper, and the fact that it is cartesian is by \cite[Observation 107]{VerityComplicialI}.
\end{proof}

\begin{rmk}
\begin{enumerate}
    \item The argument employed to show (ii) is standard and holds more generally in categories of presheaves as treated in \cref{cisinskiappendix} (cf~\cref{lemmalocalobjects}).
    \item The model structure for $n$-complicial sets built in \cref{modelstructureonstrat} coincides with the $\strat$-enriched Bousfield localization of the model structure for $n$-trivial (weak) complicial sets from \cite[Example 104]{VerityComplicialI}, with respect to the saturation anodyne extensions. The ingredients needed to recognize the fibrant objects in such localized model structure as the $n$-complicial sets are precisely the techniques from \cref{TechnicalResultsAnodyne}.
\end{enumerate}
\end{rmk}

\subsection{The model structure for $n$-precomplicial sets}

We now construct a Cisinski model structure on $\pstrat$ for $n$-precomplicial sets.

\begin{thm}
\label{modelstructureondiscretepresheaves}
For $n\ge0$, the category $\pstrat$ supports a cartesian model structure where 
\begin{itemize}
\item the cofibrations are precisely the monomorphisms,
\item the fibrant objects are precisely the $n$-precomplicial sets, and
\item the weak equivalences are precisely $\Lambda_n$-local weak equivalences.
\end{itemize}
\end{thm}

The key fact to check is that $\Lambda_n$ interacts well with $\Delta[1]_t$ in the sense of the following proposition.

We recall that an \emph{anodyne extension} in $\pstrat$ is a map that can be expressed as a retract of a transfinite composition of pushouts in $\pstrat$. Given that colimits in $\pstrat$ and in $\strat$ differ in general, for a map of stratified simplicial sets being anodyne in $\strat$ or in $\pstrat$ is not a priori the same requirement.

\begin{prop}
\label{anodyneproperties}
The elementary anodyne extensions from \cref{anodynemaps} generate a class of anodyne extensions in $\pstrat$ relative to $\Delta[1]_t$ in the sense of \cref{anodynegenerate}, ie,
\begin{enumerate}
    \item the pushout-product of either inclusion $\Delta[0]\hookrightarrow\Delta[1]_t$ and a monomorphism $K\hookrightarrow L$ and,
    $$(\Delta[0]\times L)\aamalg{\Delta[0]\times K}(\Delta[1]_t\times K)\to \Delta[1]_t\times L,$$
    is an anodyne extension, and
    \item the pushout-product of the inclusion $\partial\Delta[1]\hookrightarrow\Delta[1]_t$ and an elementary anodyne extension $I\hookrightarrow J$,
    $$(I\times\Delta[1]_t)\aamalg{I\times\partial\Delta[1]}(J\times\partial\Delta[1])\to J\times\Delta[1]_t,$$
    is an anodyne extension.
\end{enumerate}
\end{prop}

\begin{proof}[Proof of \cref{anodyneproperties}]
We verify the conditions (1)-(2).
\begin{enumerate}
    \item We first observe that the pushout-product of either inclusion $\Delta[0]\hookrightarrow\Delta[1]_t$ (which is in particular a complicial horn inclusion) and any monomorphism $K\hookrightarrow L$ is an anodyne extension in $\strat$ by \cref{AnodyneTimesMonoStrat}. 
    By \cref{anodyneinStratisanodyneinpStrat} we conclude that the same map is also an anodyne extension in $\pstrat$.
    \item We know from \cref{AnodyneTimesMonoStrat} that the pushout-product of the inclusion $\partial\Delta[1]\hookrightarrow\Delta[1]_t$ and a generating anodyne map is an anodyne extension in $\strat$. 
    By \cref{anodyneinStratisanodyneinpStrat} we conclude that the same map is also an anodyne extension in $\pstrat$.
    \qedhere
    \end{enumerate}
\end{proof}

We can now prove the theorem.

\begin{proof}[Proof of \cref{modelstructureondiscretepresheaves}]
The existence of the desired model structure is an application of \cref{cisinskimodelstructure} for model structures on $\psh{t\Delta}$. In order to use the theorem, we only need to know that the class $\Lambda_n$ generates a class of anodyne extensions relative to $\Delta[1]_t$, which was proven in \cref{anodyneproperties}.

We now argue that the model structure is cartesian. By \cite[Prop. 2.7]{Ara} it is enough to show that
the pushout-product of any anodyne extension $I\to J$ in $\psh{t\Delta}$ and any monomorphism $K\to L$ in $\psh{t\Delta}$,
$$(I\times L)\aamalg{I\times K}(J\times K)\to J\times L,$$
is an anodyne extension in $\psh{t\Delta}$. 
Without loss of generality, we can assume that $I\to J$ is an elementary anodyne extension from \cref{anodynemaps} and $K\to L$ is a generating monomorphism, as in \cref{generatingmonos}; in particular both maps, as well as their pushout-product, live in $\strat$. By \cref{retractlemma}, it is enough to show that the pushout-product is an anodyne extension in $\strat$ (rather than in $\psh{t\Delta}$).
The conclusion now follows from \cref{AnodyneTimesMonoStrat}.
\end{proof}

\begin{rmk}
\label{generalcomments}
There are strictly more $n$-precomplicial sets than $n$-complicial sets. However, a stratified simplicial set is fibrant if and only if it is an $n$-complicial set.
\end{rmk}

\subsection{The Quillen equivalence}
\label{The Quillen equivalence of pStrat and Strat}

We can now show that the model structure for $n$-complicial sets and that for $n$-precomplicial sets are Quillen equivalent.

\begin{prop}
\label{Quillenpair}
The reflector $R$ preserves cofibrations and weak equivalences. In particular, the adjunction
$$R\colon\pstrat\rightleftarrows\strat\colon i$$
is a Quillen pair between the model structures for $n$-precomplicial sets and $n$-complicial sets.
\end{prop}

\begin{proof}
The fact that the reflector $R$ preserves cofibrations is straightforward from the explicit description of $R$. We now prove that the reflector $R$ respects weak equivalences, too.

If $X\to Y$ is a weak equivalence in $\pstrat$, then $\ihom{Y}{W}\to\ihom{X}{W}$
is a $\Delta[1]_t$-homotopy equivalence in $\pstrat$ for any $n$-precomplicial set $W$, and in particular for any $n$-complicial set $W$. Using the fact that the adjunction is enriched from \cref{enrichedadjunction}, we obtain  that
$\ihom{RY}{W}\to\ihom{RX}{W}$
is a homotopy equivalence in $\pstrat$ for any $n$-complicial set $W$, and therefore $RX\to RY$ is a weak equivalence in $\strat$.
\end{proof}

\begin{prop}
\label{reflectorrespectsfibrant}
The reflector respects fibrant objects.
\end{prop}

The proof makes use of the following preliminary lemma.

\begin{lem}
\label{unittrivialfibration}
For any prestratified simplicial set $X$ the unit $X\to iRX$ is an acyclic fibration.
\end{lem}

\begin{proof}[Proof of \cref{unittrivialfibration}]
We prove that for any presheaf $X$ the unit $X\to iRX$ lifts against all cofibrations, which are precisely the monomorphisms.
By \cref{generatingmonos}, it suffices to show that the unit lifts against each of the generating monomorphisms 
$$\partial\Delta[m]^{\flat}\hookrightarrow\Delta[m]^{\flat}\text{ and }\Delta[m]\hookrightarrow\Delta[m]_t.$$

First, consider a lifting problem in $\pstrat$ of the form 
\[
\begin{tikzcd}
\partial \Delta[m]^\flat \arrow[r]\arrow[d, hook] & X \arrow[d]\\
\Delta[m]^\flat \arrow[r] \arrow[ur, dashed, "?"] & iRX
\end{tikzcd}
\]
for $m\geq 0$.
Since $(-)^{\flat}$ is left adjoint to the forgetful functor $U$, this lifting problem is equivalent to the following one in $\sset$
\[
\begin{tikzcd}
\partial \Delta[m] \arrow[r]\arrow[d, hook] & UX \arrow[d,"\cong"]\\
\Delta[m] \arrow[r] \arrow[ur, dashed, "?"] & UiRX,
\end{tikzcd}
\]
which admits a solution. 

Next, consider a lifting problem in $\pstrat$ of the form
\[
\begin{tikzcd}
 \Delta[m] \arrow[r, "g"]\arrow[d, hook] & X \arrow[d]\\
\Delta[m]_t \arrow[r] \arrow[ur, dashed, "?"] & iRX,
\end{tikzcd}
\]
for $m\ge1$.
The data of this commutative square
is precisely an $m$-simplex $\sigma\in X_m$ in the underlying simplicial set of $X$ that is marked as an $m$-simplex of $RX$, namely $\sigma\in t(RX)_m$, and the lift exists if and only if there exists a marked simplex $\tilde\sigma\in tX_m$ that is mapped to $\sigma$ by the structure map $tX_m\to X_m$.
By definition of $R$, the map
$$tX_m\twoheadrightarrow\im(\varphi\colon \thin{X}{m}\to X_m)=\thin{(iRX)}{m}$$
is surjective, and an elementary diagram chase shows that any preimage $\tilde\sigma\in tX_m$ of $\sigma$ gives the desired solution.
\end{proof}

We can now prove the proposition.

\begin{proof}[Proof of \cref{Quillenpair}]
Let $A$ be a fibrant object in $\pstrat$.
In order to show the fibrancy of $RA$ in $\strat$, we need to solve a lifting problem of the following form
\[
\begin{tikzcd}
J \arrow[d, hook]\arrow[r]& RA\\
J'\arrow[ur, dashed, "?" swap]&
\end{tikzcd}
\]
for $J\hookrightarrow J'$ any elementary anodyne extension in $\strat$ (and thus in $\pstrat$).
By reading the diagram in $\pstrat$ instead, we look at a lifting problem as follows
\[
\begin{tikzcd}
iJ \arrow[d, hook]\arrow[r]& iRA.\\
iJ'\arrow[ur, dashed, "?" swap]&
\end{tikzcd}
\]
We add to the diagram the unit of $A$, which is an acyclic fibration by \cref{unittrivialfibration}, obtaining
\[
\begin{tikzcd}
& A \arrow[d, twoheadrightarrow, "\eta"]\\
iJ \arrow[d, hook]\arrow[r]& iRA.  \\
iJ'\arrow[ur, dashed, "?" swap] & 
\end{tikzcd}
\]
Since $iJ$ is cofibrant, we find a lift $iJ\to A$ as follows
\[
\begin{tikzcd}
& A \arrow[d, twoheadrightarrow, "\eta"]\\
iJ \arrow[d, hook]\arrow[ur]\arrow[r]& iRA.  \\
iJ'\arrow[ur, dashed, "?" swap] & 
\end{tikzcd}
\] 
Since the inclusion $iJ\to iJ'$ is an acyclic cofibration, we find a further lift $iJ'\to A$ in $\pstrat$
\[
\begin{tikzcd}
& A \arrow[d, twoheadrightarrow, "\eta"]\\
iJ \arrow[d, hook]\arrow[ur]\arrow[r]& iRA.  \\
iJ'\arrow[ur, dashed, "?" swap] \arrow[uur, dotted]& 
\end{tikzcd}
\]
By composing with the unit, we obtain the desired lift.
\end{proof}

\begin{prop}
\label{inclusioncreatesequivalences}
The inclusion creates, namely reflects and preserves, weak equivalences.
\end{prop}

\begin{proof}
Let $X\to Y$ be a map in $\strat$. The map $iX\to iY$ is a weak equivalence  in $\pstrat$ if and only if the map
$$\hom_{\pstrat}(iY,A)\to\hom_{\pstrat}(iX,A)$$
is a $\Delta[1]_t$-homotopy equivalence for any $A$ fibrant in $\pstrat$. Since the model structure is cartesian by \cref{modelstructureondiscretepresheaves}, 
the unit is an acyclic fibration by \cref{unittrivialfibration}, and acyclic fibrations are $\Delta[1]_t$-homotopy equivalences by \cite[Prop. 1.3.26]{cisinski}, this is equivalent to asking for the map
$$\hom_{\pstrat}(iY,iRA)\to\hom_{\pstrat}(iX,iRA)$$
to be a $\Delta[1]_t$-homotopy equivalence in $\psh{t\Delta}$. By \cref{inclusionfullyfaithful}, this is equivalent to asking for the map
$$\hom_{\strat}(Y,RA)\to\hom_{\strat}(X,RA)$$
to be a $\Delta[1]_t$-homotopy equivalence in $\strat$. By \cref{reflectorrespectsfibrant} and since the reflector is essentially surjective (even when restricted to fibrant objects on both sides), 
this is equivalent to saying that $X\to Y$ is a weak equivalence in $\strat$.
\end{proof}

\begin{prop}
\label{Quillenequivalence}
For $n\ge0$, the adjunction
$$R\colon\pstrat\rightleftarrows\strat\colon i$$
is a Quillen equivalence between the model structure for~$n$-precomplicial sets and $n$-complicial sets.
\end{prop}

The proof makes use of the following characterization of Quillen equivalence that applies to our situation.

\begin{prop}[{\cite[Prop.3.5]{FKKR}}]
\label{QuillenequivalenceRcreates}
If in a Quillen pair the right adjoint creates weak equivalences and the unit of any cofibrant object is a weak equivalence, the Quillen pair is a Quillen equivalence.\hfill{$\square$}
\end{prop}

\begin{proof}[Proof of \cref{Quillenequivalence}]
The inclusion creates weak equivalences by \cref{inclusioncreatesequivalences} and the unit is a weak equivalence by \cref{unittrivialfibration}.
We conclude by \cref{QuillenequivalenceRcreates}.
\end{proof}

\subsection{The role of saturation}
\label{sectionsaturation}
We collect here for the interested reader a few remarks that should clarify what difference it makes to include the saturation anodyne extension in the list from \cref{anodynemaps}.

Verity \cite{VerityComplicialI} originally focused on a model structure on $\strat$ built in the same way as the one from \cref{modelstructureonstrat} but without including the saturation anodyne extension.
This translates to saying that the fibrant objects do not have the right lifting property with respect to this class of saturation anodyne extensions, and are therefore not (necessarily) \emph{saturated}.

It is a natural question to wonder in what Verity's original model structure, which we will refer to for the time being as the model structure for \emph{non-saturated} $n$-complicial sets\footnote{In Verity's terminology, these would be weak complicial sets that are moreover $n$-trivial.} and the model structure for $n$-complicial sets from \cref{modelstructureonstrat}.

More precisely, we would like to answer the following natural questions.
\begin{enumerate}[ref=(\arabic*)]
    \item\label{Q1} Are the two model structures equal?
    \item\label{Q2} If not, are the homotopy theories represented by these model structures equivalent?
    \item\label{Q3} What differences arise when focusing on stratified simplicial sets coming from the familiar examples of strict $n$-categories?
\end{enumerate}

The answer to Question \ref{Q1} is easily formulated.

\begin{answ}
Let $n>0$. The model structure for non-saturated $n$-complicial sets and the one for $n$-complicial sets are not equal.
\end{answ}

To see this, let us denote by $\mathbb I$ the free isomorphism category (regarded as a discrete $n$-category if $n>1$), and by $N^{RS}(\mathbb I)$ its Roberts--Street nerve, ie, the stratified simplicial set whose underlying simplicial set is the ordinary nerve $N(\mathbb I)$, and whose marked simplices are those corresponding to identities. In particular, in $N^{RS}(\mathbb I)$ everything is marked in dimension higher or equal than two, and only degenerate $1$-simplices are marked.

\begin{prop}
Let $n>0$.
\begin{enumerate}[leftmargin=*, label=(\alph*), ref=(\alph*)]
    \item\label{NonSat} In the model structure for non-saturated $n$-complicial sets, the Roberts--Street nerve $N^{RS}(\mathbb I)$ is fibrant and the unique map $N^{RS}(\mathbb I)\to\Delta[0]$ is not a weak equivalence.
    \item\label{Sat} In the model structure for $n$-complicial sets, the Roberts--Street nerve $N^{RS}(\mathbb I)$ is not fibrant and the unique map $N^{RS}(\mathbb I)\to\Delta[0]$ is a weak equivalence.
\end{enumerate}
\end{prop}

\begin{proof}
For \ref{NonSat}, we know by classical results eg\ \cite{StreetFillers,SteinerComplicialNerves,VerityComplicialAMS}, 
that the stratified simplicial set $N^{RS}(\mathbb I)$ has the right lifting property with respect to all complicial horns, thinness and triviality anodyne extensions from \cref{anodynemaps}, and is therefore a non-saturated $n$-complicial set. Next, we observe that for $N^{RS}(\mathbb I)\to\Delta[0]$ to be a weak equivalence in the first model structure, it would have to be an acyclic fibration and in particular lift against the monomorphism $\Delta[1]\hookrightarrow\Delta[1]_t$. This means that all $1$-simplices would have to be marked, which is not true.

For \ref{Sat}, consider the $3$-simplex $(f^{-1},f,f^{-1})$ of the nerve of $\mathbb I$, $f$ being one of the two non-identity morphisms of $\mathbb I$. 
This $3$-simplex defines a map $\Delta[3]_{eq}\to N^{RS}(\mathbb I)$, which
cannot be lifted along the saturation anodyne extension $\Delta[3]_{eq}\to\Delta[3]^{\sharp}$, so $N^{RS}(\mathbb I)$ is not an $n$-complicial set. Next, we observe that the canonical map $N^{RS}(\mathbb I)\to N(\mathbb I)^{\sharp}$ can be written as a pushout of the saturation anodyne extension  $\Delta[3]_{eq}\to\Delta[3]^{\sharp}$ and is in particular a weak equivalence. Moreover, any inclusion $\Delta[0]\to N(\mathbb I)^{\sharp}$ is by \cref{KanAcyclicCof} a weak equivalence, and therefore so is the unique map $N(\mathbb I)^{\sharp}\to\Delta[0]$.
\end{proof}

The same proposition allows us to answer Question \ref{Q2}, too.

\begin{answ}
There can be no equivalence of homotopy theories between the saturated model structure and the non-saturated model structure that fixes the homotopy type of $\Delta[0]\cong N^{RS}([0])$ and the homotopy type of $N^{RS}(\mathbb I)$. 
\end{answ}

Indeed, we saw that the homotopy types of $\Delta[0]$ and $N^{RS}(\mathbb I)$ are equal in the homotopy category obtained from the model structure for $n$-complicial sets and not in the one obtained from the model structure for non-saturated $n$-complicial sets.

This fact is not a complete answer to Question \ref{Q2}, but it is at least convincing that if the homotopy theories of non-saturated $n$-complicial sets and of $n$-complicial sets are equivalent it would have to be through a quite odd comparison map that does not implement our intuition on strict $n$-categories.

We now focus on Question \ref{Q3}, and discuss how strict $n$-categories react to each of the two model structures. As we already did in the counterexample above, we consider the Roberts--Street nerve $N^{RS}$ as the natural way to regard a strict $n$-category as a stratified simplicial set. For a general $n$-category $\cC$, the Roberts--Street nerve is a stratified simplicial set whose underlying simplicial set is the Street nerve $N^{Street}(\cC)$, and whose marked $k$-simplices are those corresponding to an identity $k$-cell.
Modulo this identification, we can address Question \ref{Q3}.

\begin{answ} 
Let $n>0$.
\begin{enumerate}[leftmargin=*, label=(\alph*), ref=(\alph*)]
    \item The model structure for non-saturated $n$-complicial sets is suited to incorporate the homotopy theory of strict $n$-categories up to isomorphism.
    \item The model structure for $n$-complicial sets is suited to incorporate the homotopy theory of strict $n$-categories up to $n$-categorical equivalence.
\end{enumerate}
\end{answ}

A way to make this idea into a mathematical statement is the following.

\begin{prop}
Let\footnote{This restriction is not needed for \ref{ReflectNonSat}, and even \ref{ReflectSat} conjecturally holds true for all $n\ge0$.} $n=1,2$, and let $F\colon\cC\to\cD$ be a functor of $n$-categories.
\begin{enumerate}[label=(\alph*), ref=(\alph*)]
    \item\label{ReflectNonSat} The $n$-functor $F$ is an isomorphism of $n$-categories if and only if $N^{RS}(F)$ is a weak equivalence in the model structure for non-saturated $n$-complicial sets.
    \item \label{ReflectSat} The $n$-functor $F$ is an $n$-categorical equivalence if and only if $N^{RS}(F)$ is a weak equivalence in the model structure for $n$-complicial sets.
\end{enumerate}
\end{prop}

\begin{proof}
For \ref{ReflectSat}, we show in \cite[Theorem 4.12]{Nerves2Cat} that the model structure on the category $n\cat$ of strict $n$-categories, namely the canonical model structure on $\cat$ for $n=1$ and the Lack model structure on $2\cat$ for $n=2$, is transferred along a right adjoint from the model structure for $n$-precomplicial sets on $\pstrat$ from \cref{modelstructureondiscretepresheaves}. In particular, the right adjoint $N^{\natural}\colon n\cat\to\pstrat$ creates weak equivalences. Since we constructed in \cite[Theorem 5.2]{Nerves2Cat} a natural weak equivalence $N^{RS}(\cC)\to N^{\natural}(\cC)$ for any $n$-category $\cC$, the Roberts--Street nerve $N^{RS}$ also creates weak equivalences, as desired.

An easier variant of the same argument allows us to show \ref{ReflectNonSat}. More precisely, one can show that the trivial model structure on the category $n\cat$ of strict $n$-categories, namely the one in which weak equivalences are the isomorphisms and all maps are both fibrations and cofibrations, is transferred along the Roberts--Street nerve $N^{RS}$ from the model structure for non-saturated $n$-precomplicial sets on $\pstrat$. In particular, the right adjoint $N^{RS}$ creates weak equivalences, as desired.
\end{proof}

\section{The model structure for $n$-precomplicial spaces}
\label{section2}

In this section, we define the category $\spsh{t\Delta}$ of prestratified simplicial spaces, we construct a model structure by taking the left Bousfield localization of the injective model structure with respect to the elementary anodyne maps from \cref{anodynemaps}, and show that this model category is Quillen equivalent to that on $\pstrat$. In particular, the fibrant objects in $\spsh{t\Delta}$, which we call ``$n$-precomplicial spaces'', can be thought of as a model of $(\infty,n)$-categories.

\subsection{Prestratified simplicial spaces}

We now consider presheaves over $t\Delta$ valued in $\sset$.

\begin{defn}
A \emph{prestratified simplicial space} is a presheaf $X\colon t\Delta^{\op}\to\sset$. We denote by $\spsh{t\Delta}$ the category of prestratified simplicial spaces.
\end{defn}

The following is formal and generalizes the adjunction from \cite[Proposition 4.7]{JT}.

Going from the category $\pstrat$ of discrete presheaves over $t\Delta$ to the category $\spsh{t\Delta}$ of simplicial presheaves over $t\Delta$ means that now a presheaf $W\colon t\Delta^{\op}\to\sset$ encodes \emph{spaces} (as opposed to sets) $W_0$ of objects, $W_k$ of $k$-morphisms and $tW_k$ of $k$-equivalences. In the next section we will impose conditions in order for $W$ to model an $(\infty,n)$-category. We first investigate the categorical structure of $\spsh{t\Delta}$.

\begin{rmk}
\label{pandq}
The map $i_0\colon t\Delta\to t\Delta\times \Delta$, given by $i_0([m]_{(t)})=([m]_{(t)}, [0])$, and the projection $p\colon t\Delta\times \Delta\to t\Delta$ onto the first component induce two functors that form an adjoint pair
$$p^*\colon\psh{t\Delta} \rightleftarrows\spsh{t\Delta}\colon i_0^*.$$
The functor $p^*$ takes a prestratified simplicial set to a constant prestratified simplicial space. The functor $i_0^*$ remembers only the $0$-simplices of every simplicial set, constituting a prestratified simplicial space.

Similarly, the projection $q\colon t\Delta\times \Delta\to\Delta$ onto the second component induces a functor
$$q^*\colon\sset\to\spsh{t\Delta},$$
which takes a simplicial set to a constant prestratified simplicial space.

We refer the reader to \cite[\textsection 2.15]{Ara} for more details.
\end{rmk}

\begin{lem}[{\cite[Definition 11.7.2]{Hirschhorn}}]

\label{mappingobjects}
The category of functors $\spsh{t\Delta}$ is enriched over $\sset$. Given objects $X$ and $Z$ of $\spsh{t\Delta}$, the mapping object in $\sset$ is given componentwise by
$$\Map_{\spsh{t\Delta}}(X,Z)_m:=\Hom_{\spsh{t\Delta}}(X\times q^*\Delta[m], Z)$$
for any $m\geq 0$. The simplicial structure comes from the cosimplicial structure of the object $\Delta[\bullet]$.
\end{lem}

\subsection{Precomplicial spaces}

The notion of an $n$-precomplicial space is defined in terms of the class $p^*(\Lambda_n)$ of elementary anodyne maps and of mapping spaces.

\begin{rmk}
By \cref{tDeltasuperReedy}, the category $t\Delta$ is a regular skeletal Reedy category, and it is in particular elegant Reedy by \cref{regularandReedy}.
Thus, by \cite[Proposition 3.15]{elegant}, the Reedy cofibrations in $\spsh{t\Delta}$ turn out to be precisely the monomorphisms, so that the Reedy model structure turns out to be the injective model structure. Here, the category $\sset$ is endowed with the Kan-Quillen model structure.
\end{rmk}

Recall the set $\Lambda_n$ of elementary anodyne extensions from \cref{anodynemaps}.

\begin{defn}
An \emph{$n$-precomplicial space} is a prestratified space that is injectively fibrant and local with respect to all maps in $p^*\Lambda_n$, ie, for any elementary anodyne extension $I\to J$ in $\Lambda_n$ the induced map
$$\Map_{\spsh{t\Delta}}(p^*J,W)\to\Map_{\spsh{t\Delta}}(p^*I,W)$$
is a weak equivalence of simplicial sets.
\end{defn}

According to the intuition above, the fact that a prestratified simplicial space is local with respect to
\begin{enumerate}
    \item the horn inclusion $\Lambda^1[2]\to\Delta[2]_t$ is similar to the Segal condition, which encodes the fact that $1$-simplices can be composed, and a witness for the composite is obtained by filling the horn to a marked $2$-simplex.
    \item the horn inclusion $\Delta[0]\to\Delta[1]_t$ and the saturation anodyne extension $\eqDelta\to\Delta[3]^\sharp$ is equivalent when $n=1$ to being local with respect to the inclusion $\Delta[0]\to N\mathbb I$ of a point into the nerve of the walking isomorphism, and is therefore reminiscent of the completeness condition.
    \item the $n$-triviality anodyne extension $\Delta[l]\to\Delta[l]_t$ for $l>n$ says that in dimensions $l>n$ the space of $l$-equivalences is equivalent to the space of all $l$-simplices, namely all $l$-simplices are equivalences.
\end{enumerate}
It is interesting how the Segal and the completeness conditions are essentially both recorded by a complicial horn anodyne extension.

\subsection{The model structure on $\spsh{t\Delta}$}
We now construct a model structure on $\spsh{t\Delta}$ for $n$-precomplicial spaces as a left Bousfield localization of the vertical injective model structure.

The weak equivalences, which we call \emph{$\Lambda_n$-local equivalences}, are defined in terms of the class of $n$-precomplicial spaces and mapping spaces.

\begin{defn}
A map $g\colon C\to D$ in $\spsh{t\Delta}$ is a \emph{$\Lambda_n$-local weak equivalence} if for every $n$-precomplicial space $X$ the induced map
$$g^*\colon\Map_{\spsh{t\Delta}}(D,X)\to\Map_{\spsh{t\Delta}}(C,X)$$
is a weak equivalence in $\sset$. 
\end{defn}

\begin{thm}
\label{interestingmodelstructure}
For $n\ge0$, the category $\spsh{t\Delta}$ supports a model structure where
\begin{itemize}
\item the cofibrations are precisely the monomorphisms;
\item the fibrant objects are precisely the $n$-precomplicial spaces;
\item the weak equivalences are the vertical $\Lambda_n$-local weak equivalences;
\item the weak equivalences between fibrant objects are the levelwise weak equivalences.
\end{itemize}
\end{thm}

The proof makes use of the language of ``localizers'', and we refer the reader to \cref{Model structure on simplicial presheaves and localizers} for more details.
The key fact 
is that localizers in a presheaf category that are ``accessible'' define precisely the classes of weak equivalences for a cofibrantly generated model structure whose cofibrations are the monomorphisms.

\begin{proof}[Proof of \cref{interestingmodelstructure}]\
As in \cite[\textsection 4.1]{Ara}, let us consider the localizer
$W_{fRezk}$
generated by
$$W_{inj}\cup p^*(\Lambda_n)\cup p^*\{\partial_X^\varepsilon\colon X\to X\times\Delta[1]_t\ |\ X\in\pstrat,\varepsilon=0,1\}.$$
Since $t\Delta$ is a skeletal Reedy category by \cref{tDeltasuperReedy}, we will see from \cref{lemmaWRezk} and \cref{localization} 
that the localizer $W_{fRezk}$ is an accessible localizer. 
The localizer $W_{fRezk}$ is the class of weak equivalences for a model structure on $\spsh{t\Delta}$, referred to as the \emph{model structure of formal Rezk spaces} in \cite[\textsection4.3]{Ara}, whose cofibrations are the monomorphisms.

We claim that the localizer $W_{fRezk}$ is in fact generated by the smaller class $W_{inj}\cup p^*(\Lambda_n)$, and prove this separately as \cref{lemmaWRezk}.
Then, by \cref{localization}, the model structure of formal Rezk spaces is the Bousfield localization of the injective model structure with respect to $\Lambda_n$. In particular, the weak equivalences are precisely the $p^*(\Lambda_n)$-local weak equivalences, and the fibrant objects are precisely the $n$-precomplicial sets, as desired.
\end{proof}

We complete the proof of the theorem by proving the claim.

\begin{lem}
\label{lemmaWRezk}
The localizer $W_{fRezk}$ is generated by the class $W_{inj}\cup p^*(\Lambda_n)$.
\end{lem}

\begin{proof}
By definition, $W_{fRezk}$ is generated by
$$W_{inj}\cup p^*(\Lambda_n)\cup p^*\{X\to X\times\Delta[1]_t\ |\ X\in\pstrat\}.$$
As explained in the proof of \cite[Proposition 4.2]{Ara}, we know that $W_{fRezk}$ is generated by
\begin{align*}
&W_{inj}\cup p^*(\Lambda_n)\cup\quad\quad\\
&\quad\cup p^*\{\partial\Delta[m]\times\Delta[1]_t \aamalg{\partial\Delta[m]\times \Delta[0]}\Delta[m]\times \Delta[0]\to \Delta[m]\times\Delta[1]_t \, |\, m\geq 0\} \\
&\quad\quad\quad\cup p^*\{\Delta[m]\times\Delta[1]_t \aamalg{\Delta[m]\times \Delta[0]}\Delta[m]_t\times \Delta[0]\to\Delta[m]_t\times\Delta[1]_t \, |\, m\geq 1\}.
\end{align*}
As observed in \cref{anodyneproperties}, the pushout-product of a generating monomorphism $K\to L$ and either inclusion $\Delta[0]\to\Delta[1]_t$,
$$K\times\Delta[1]_t \aamalg{K\times \Delta[0]}L\times \Delta[0]\to L\times\Delta[1]_t,$$ lies in the saturation of $\Lambda_n$, in the sense of \cref{saturation}. Given that $p^*$ preserves all colimits as a left adjoint, the pushout product
$$p^*(K\times\Delta[1]_t \aamalg{K\times \Delta[0]}L\times \Delta[0])\to p^*(L\times\Delta[1]_t)$$
lies in the saturation of $p^*(\Lambda_n)$.
Thus, this class generates the same localizer as  $W_{inj}\cup p^*(\Lambda_n)$.
\end{proof}

\begin{digression}
The model structure on $\spsh{t\Delta}$ for $n$-precomplicial spaces can also be realized as a localization of the injective model structure on $\Fun(\Delta^{\op},\psh{t\Delta})\cong \spsh{t\Delta}$ (sometimes referred to as the ``horizontal (injective) model structure''). More details on this description can be found in \cite{Ara}.
\end{digression}

\begin{rmk}
The adjunction
    $$(-)^{\flat}\colon\spsh{\Delta}\rightleftarrows\spsh{t\Delta}\colon U$$
    is a Quillen equivalence
\begin{enumerate}
    \item[(0)] between the model structure for $0$-complicial spaces and the model structure for $\infty$-groupoids.
    \item between the model structure for $1$-complicial spaces and the Rezk model structure for complete Segal spaces.
\end{enumerate}
This pattern cannot be extended to $n>1$ because, unlike $1$-equivalences, $2$-equivalences cannot be detected just by means of the simplicial structure.
\end{rmk}

\begin{rmk}
A prestratified simplicial space $W$ which is an $n$-precomplicial space is morally ``stratified'', in the sense that the structure maps
$$tW_m\to W_m$$
are injective in a homotopical sense. Let us illustrate this for the case $m=2$.

We observe that there is a commutative diagram in $\pstrat$
\[
\begin{tikzcd}[column sep=0.5cm]
\Delta[1]\aamalg{\Delta[0]}\Delta[1]\arrow[d, "\cong"]\arrow[rr, "\simeq"]&&\Delta[2]_t \arrow[d, equal]\\
\Lambda^1[2]\arrow[rr, "\simeq"] \arrow[rd]&& \Delta^1[2]\\
&\Delta[2]\arrow[ur]&,
\end{tikzcd}
\]
where the horizontal maps are acyclic cofibrations.
When applying $p^*$ and taking mapping spaces into $W$, we obtain a commutative diagram in $\sset$ 
\[
\begin{tikzcd}[column sep=-1cm]
\Map(p^*\Delta[1], W) \underset{\Map(p^*\Delta[0], W)}{\times} \Map(p^*\Delta[1], W)&&\Map(p^*\Delta[2]_t, W) \arrow[d, equal]\arrow[ll, "\simeq"]\\
\Map(p^*\Lambda^1[2], W) \arrow[u, "\cong"] && \Map(p^*\Delta^1[2], W)\arrow[ll, "\simeq"]\arrow[ld]\\
&\Map(p^*\Delta[2], W)\arrow[ul]&,
\end{tikzcd}
\]
where the horizontal maps are now acyclic fibrations, and in particular weak equivalences. This can be read as saying that the structure map 
$$tW_2=W^{p^*\Delta[2]_t}\to W^{p^*\Delta[2]}=W_2$$
has a left inverse up to homotopy, as desired.
\end{rmk}

\subsection{Two Quillen equivalences with $\pstrat$}

We now show that the model structures for $n$-precomplicial sets and $n$-precomplicial spaces are Quillen equivalent in two different ways. This is analogous to the more familiar picture from \cite[\textsection 4]{JT}, where Joyal--Tierney exhibit two Quillen equivalences between the model structure for quasi-categories and the model structure for complete Segal spaces.

\begin{thm}
\label{QE1}
For $n\ge0$, the adjunction
$$p^*\colon\pstrat \rightleftarrows\spsh{t\Delta} \colon i_0^*$$
is a Quillen equivalence between the model structures for $n$-precomplicial sets and $n$-precomplicial spaces.
\end{thm}

\begin{proof}
The proof is an application of a variant of \cite[Theorem 4.11(1)]{Ara}, which states the desired Quillen equivalence, given that $t\Delta$ was shown to be a regular skeletal Reedy category in \cref{tDeltasuperReedy}.

In order to evoke the original formulation of \cite[Theorem 4.11]{Ara} we would need to have that the cylinder object $\Delta[1]_t$ is fibrant, which is not the case. However, a careful analysis of the argument shows that, in order to prove \cite[Theorem 4.11]{Ara}, this assumption is only used to show the preliminary result \cite[Theorem 2.14]{Ara} and to prove \cite[Theorem 4.10]{Ara}. For these to hold in our situation, it suffices to know that the projection $X\times \Delta[1]_t \to X$ is a weak equivalence for every prestratified simplicial set $X$. 
The projection $X\times \Delta[1]_t \to X$ is a retraction of either of the two canonical inclusions $X\to X\times \Delta[1]_t$, which can be seen as pushout-products of the map $\Delta[0]\to\Delta[1]_t$ and the identity of $X$; by \cref{anodyneproperties} these pushout-products are weak equivalences, and therefore so is the projection $X\times \Delta[1]_t \to X$.
\end{proof}

\begin{rmk}
As a special case of \cite[\textsection2.22]{Ara}, the cosimplicial object $\Delta[\bullet]^\sharp$ in $\pstrat$
induces an adjunction 
$$Real\colon \spsh{t\Delta} \rightleftarrows\psh{t\Delta}\colon Sing.$$
\end{rmk}

\begin{thm}
\label{secondQuillenequivalence}
For $n\ge0$, the adjunction
$$Real\colon \spsh{t\Delta} \rightleftarrows\psh{t\Delta}\colon Sing$$
is a Quillen equivalence between the model structures for $n$-precomplicial spaces and $n$-precomplicial sets.
\end{thm}

In order to apply Ara's machinery, we need some preliminary work. Recall from \cref{sharpandflat} that $X^\sharp$ denotes the maximal stratification of a simplicial set $X$.

\begin{prop}
\label{cosimplicialresolution}
The cosimplicial object $\Delta[\bullet]^\sharp$ in $\pstrat$ is a \emph{cosimplicial $W(\Lambda_n)$-resolution} in the sense of \cite[\textsection 2.22]{Ara}, ie,
\begin{enumerate}
    \item the canonical morphism $\Delta[0]^{\sharp}\amalg\Delta[0]^{\sharp}\to\Delta[1]^{\sharp}$ is a 
monomorphism;
    \item for every $k\ge0$ and every $X\in\pstrat$, the canonical projection $X\times\Delta[k]^\sharp\to X$
is a $W(\Lambda_n)$-equivalence, namely a weak equivalence in $\pstrat$.
\end{enumerate}
\end{prop}

The proof of the proposition makes use of the following lemma.

\begin{lem}\label{KanAcyclicCof} 
For $n\ge0$, if a map of simplicial sets $X\to Y$ is an acyclic Kan cofibration, then $X^{\sharp} \to Y^{\sharp}$ is a $\Lambda_n$-anodyne extension in $\pstrat$.
\end{lem}

\begin{proof}
We first recall that any acyclic Kan cofibration is a retract of a transfinite composition of pushouts of horn inclusions, and that the maximal stratification $(-)^{\sharp}$ commutes with colimits. Thus, if $X\to Y$ is a retract of some 
$X_0\to \colim X_i$ with $X_i \to X_{i+1}$ being a pushout of a horn inclusion, then $X^{\sharp} \to Y^{\sharp}$ is a retract of $X_0^{\sharp}\to \colim (X_i^{\sharp})$. Without loss of generality, we can therefore assume that $X\to Y$ is a pushout of a horn inclusion of the form
\[
\begin{tikzcd}
\Lambda^j[m] \arrow[d, hook]\arrow[r, "f"]& X \arrow[d]\\
\Delta[m]\arrow[r]& Y,
\end{tikzcd}
\]
and show that $X^{\sharp} \to Y^{\sharp}$ is a $\Lambda_n$-anodyne extension in $\pstrat$.

If $P$ denotes the pushout in $\pstrat$ of $f$ along a horn anodyne extension as displayed
\[
\begin{tikzcd}
\Lambda^j[m] \arrow[d, hook]\arrow[r, "f"]& X^{\sharp} \arrow[d]\\
\Delta^j[m]\arrow[r, "f'" swap]& P,
\end{tikzcd}
\]
then $X^\sharp \to Y^\sharp$ factors as
$$X^\sharp\to P\to Y^\sharp.$$
While the first map is a $\Lambda_n$-anodyne extension in $\pstrat$ by definition, we now argue that also the second map is one.

A direct verification shows that the map $P\to Y^\sharp$ is an entire inclusion, and that there is exactly one simplex in $P$ that is marked in $Y$ and not in $P$, namely the $j$-th face of the $m$-simplex $f'\colon \Delta^j[m]\to P$. In particular, the inclusion $P\to Y$ fits into a pushout square
\[
\begin{tikzcd}
\Delta^j[m]' \arrow[d, hook]\arrow[r, "f'"]& P \arrow[d]\\
\Delta^j[m]''\arrow[r, "f'" swap]& Y^{\sharp},
\end{tikzcd}
\]
and is in particular a $\Lambda_n$-anodyne extension in $\pstrat$.
\end{proof}

\begin{proof}[Proof of \cref{cosimplicialresolution}]
We check Conditions (1)-(2) of being a $W(\Lambda_n)$-resolution.
\begin{enumerate}
    \item The fact that Condition (1) holds is clear.
    \item The projection
$$X\times\Delta[k]^\sharp\to X\cong  X \times \Delta[0]^{\sharp},$$
is a weak equivalence if and only if the cofibration
$$X \cong X \times \Delta[0]^{\sharp} \xrightarrow{\id _X\times {[0]}} X\times\Delta[k]^{\sharp}$$
 which is a right inverse, is one.
Given that the model category $\pstrat$ is cartesian from \cref{modelstructureondiscretepresheaves}, it is enough to know that the cofibration $\Delta[0]^{\sharp}\to \Delta[k]^{\sharp}$ is a weak equivalence, and this follows from \cref{KanAcyclicCof}. This concludes the proof of Condition (2).\qedhere
\end{enumerate}
\end{proof}

We can now prove the second Quillen equivalence.

\begin{proof}[Proof of \cref{secondQuillenequivalence}]
This is an application of the same variant of \cite[Theorem 4.11]{Ara}(2) that was discussed in the proof of \cref{QE1}.
For this, we only need to know that $t\Delta$ is a regular skeletal Reedy category, which was proven in \cref{tDeltasuperReedy}, and that $\Delta[\bullet]^\sharp$ is a $W(\Lambda_n)$-resolution, which was proven in \cref{cosimplicialresolution}.
\end{proof}

\appendix

\section{Technical tools on model structures}
\addtocontents{toc}{\protect\setcounter{tocdepth}{1}}

Let $A$ be a small category, with $t\Delta$ as a motivating example. We recall the Cisinski model structure from \cite[\textsection 1.3]{cisinski} on $\psh{A}$, and the Bousfield localization of the injective model structure on $\spsh{A}$ in terms of localizers, as in \cite{Ara}.

We denote by $A[a]$ the presheaf represented by an object $a$, and by $*$ the presheaf constant at a singleton, which is terminal in $\psh{A}$.

\subsection{The Cisinski model structure on discrete presheaves}

\label{cisinskiappendix}

Consider the category $\psh{A}$ of discrete presheaves over $A$.

The following is a special case of \cite[Proposition I.6.1]{MacLaneMoerdijk} (and its proof).

\begin{prop}
\label{presheafcartesianclosed} 
The category $\psh{A}$ is cartesian closed, with internal hom between a presheaf $Y$ and a presheaf $Z$ given componentwise by
$$(\ihom{Y}{Z})_{a}:=\Hom_{\psh{A}}(A[a]\times Y,Z).$$
\end{prop}

Let $\mathbb I$ be an object of $\psh{A}$ endowed with two distinguished inclusions
 $\partial_\varepsilon\colon * \to \mathbb{I}$ for $\varepsilon=0,1$ so 
that the induced map $\partial_0\sqcup \partial_1 \colon * \sqcup * \to \mathbb{I}$
is a monomorphism. The motivating example is the object $\Delta[1]_t$ of $\psh{t\Delta}$, endowed with the two canonical inclusions $\Delta[0]\to \Delta[1]_t$. Then, as mentioned in \cite[Exemple 1.3.8]{cisinski}, the functor $\mathbb I\times-$ is a \emph{cylinder} in the sense of \cite[D\'efinition 1.3.1]{cisinski}.

This guarantees that one can define the usual notion of homotopy and homotopy equivalence.

\begin{defn}[{\cite[D\'efinition 1.3.3]{cisinski}}]
\label{I-homotopic}
Two morphisms $u_{\varepsilon}\colon X\to Y$ for $\varepsilon=0,1$ in $\psh{A}$ are called
\begin{enumerate}
    \item \emph{elementarily $\mathbb I$-homotopic} if there exists an \emph{$\mathbb I$-homotopy} between them, ie, a morphism $H\colon X\times \mathbb{I} \to Y$ so that $H\circ \partial_\varepsilon=u_\varepsilon$ for $\varepsilon=0,1$.
    \item \emph{$\mathbb I$-homotopic} if there exists a zig-zag of $\mathbb I$-homotopies between them.
    \end{enumerate}
\end{defn}

The $\mathbb I$-homotopy relation is the equivalence relation generated by the elementary $\mathbb I$-homotopy relation. 
Given that the $\mathbb I$-homotopy relation is suitably compatible with composition, 
it makes sense to define the $\mathbb{I}$-homotopy category $\Ho_{\mathbb I}(\psh{A})$ of $\psh{A}$, as the quotient of $\psh{A}$ by the $\mathbb I$-equivalence relation.

\begin{defn}[{\cite[\textsection 1.3.4]{cisinski}}]
\label{homotopyequivalence}
A morphism $u\colon X \to Y$ in $\psh{A}$ is an $\mathbb I$-homotopy equivalence if it becomes an isomorphism in the $\mathbb I$-homotopy category.
\end{defn}

The following terminology is equivalent to that from \cite[D\'efinition 1.3.21]{cisinski}. The interested reader can see the proof of the equivalences of our terminology with Cisinski's in \cref{lemmalocalobjects} and \cref{localvscisinskiequivalences}.

\begin{defn}\label{defLocal}
Let $\Lambda$ be a set of morphisms of $\psh{A}$.
\begin{itemize}
\item An object $X$ of $\psh{A}$ is \emph{$\Lambda$-local} if for any $\alpha\colon J \to J'$ in $\Lambda$, the induced map on internal homs
\[
\alpha^*\colon\ihom{J'}{X}\to\ihom{J}{X}
\]
is an acyclic fibration, ie, it has the right lifting property with respect to all monomorphisms.
\item A map $f\colon Y\to Y'$ in $\psh{A}$ is a \emph{$\Lambda$-local weak equivalence} if for every $\Lambda$-local object $X$ the map
\[
f^*\colon\ihom{Y'}{X}\to\ihom{Y}{X}
\]
is an $\mathbb I$-homotopy equivalence.
\end{itemize}
\end{defn}

Cisinski gives conditions 
on the cylinder object $\mathbb I$ and on the set $\Lambda$ for $\psh{A}$ to support a model structure. These conditions are given in terms of ``anodyne extensions''.

\begin{defn}
\label{saturation}
Let $\Lambda$ be a set of morphisms in $\psh{A}$.
The \emph{saturation} of $\Lambda$, or the \emph{class of $\Lambda$-anodyne extensions}, is the class of morphisms that have the left lifting property with respect to all morphisms having right lifting property with respect to $\Lambda$.

The small object argument
shows that a morphism is a \emph{$\Lambda$-anodyne extension} if and only if it can be written as a retract of transfinite compositions of pushouts of
morphisms of $\Lambda$.
\end{defn}

\begin{defn}
\label{generateelementaryanodyne}
\label{anodynegenerate}
Let $\Lambda$ be a set of morphisms in $\psh{A}$. The set $\Lambda$ is a \emph{set of elementary anodyne extensions}, and we say that it \emph{generates a class of anodyne extensions}, relative to $\mathbb I$, if the following conditions are met. 
\begin{enumerate}
    \item The pushout-product of a monomorphism $K\to L$ and either inclusion $\partial_\varepsilon\colon*\to\mathbb{I}$,
$$(K\times\mathbb{I})\aamalg{K\times *}(L\times *)\to L\times\mathbb I,$$
     is a $\Lambda$-anodyne extension.
    \item The pushout-product of the inclusion 
   $\partial_0\sqcup \partial_1 \colon * \sqcup * \to \mathbb{I}$ and an elementary $\Lambda$-anodyne extension $I\to J$,
$$(I\times\mathbb I) \aamalg{I\times (*\amalg*)}\left(J\times (*\amalg*)\right)\to J\times\mathbb I,$$
   is a $\Lambda$-anodyne extension.
\end{enumerate}

\end{defn}

The following theorem is a special case of \cite[Th\'eor\`eme 1.3.22]{cisinski}.

\begin{thm}
\label{cisinskimodelstructure}
If the set $\Lambda$ is a set of elementary anodyne extensions relative to $\mathbb{I}$,
then the category $\psh{A}$ supports a cofibrantly generated model structure where
\begin{itemize}
\item the cofibrations are precisely the monomorphisms.
\item the fibrant objects are precisely the $\Lambda$-local objects.
\item the weak equivalences are precisely the $\Lambda$-local maps.
\end{itemize}
We call this a \emph{Cisinski model structure} on $\psh{A}$.
\end{thm}

We prove the following characterization of local objects.

\begin{prop}
\label{lemmalocalobjects}
Let $\Lambda$ be a class of elementary anodyne extensions relative to $\mathbb{I}$. Suppose moreover that the pushout-product of any $\Lambda$-anodyne extension $I\to J$ and any monomorphism $K\to L$,
$$(I\times L)\aamalg{I\times K}(J\times K)\to J\times L,$$
is an $\mathbb I$-anodyne extension.
For a presheaf $X$ on $A$, the following are equivalent.
\begin{enumerate}
\item The object $X$ has the right lifting property with respect to $\Lambda$.
\item\label{RLPAn} The object $X$ has the right lifting property with respect to all $\Lambda$-anodyne extensions.
\item \label{LocalStrong} The object $X$ is $\Lambda$-local, ie for any $\alpha\colon J \to J'$ in $\Lambda$, the induced map on internal homs
\[
\alpha^*\colon\ihom{J'}{X} \to \ihom{J}{X}
\]
is an acyclic fibration.
    \item The object $X^Y$ is $\Lambda$-local for any $Y\in\psh{A}$.
    \label{fibrancyexponential}
\end{enumerate}
\end{prop}

\begin{proof}
The equivalence of (1) and (2) is a standard consequence of Quillen's small object argument. We now show that (2) is equivalent to (3), and that (3) is equivalent to (4).

In order to show that (2) implies (3), let $K\hookrightarrow L$ be a monomorphism in $\psh{A}$, and consider the lifting problem
\[
\begin{tikzcd}
K \arrow[r]\arrow[d, hook] & \ihom{J'}{X} \arrow[d, "\alpha^*"]\\
L\arrow[r] \arrow[ur, dashed, "?"] & \ihom{J}{X}
\end{tikzcd}
\]
By adjointness, the lifting problem is equivalent to the following one
\[
\begin{tikzcd}
K\times J' \aamalg{K\times J} L\times J \arrow[r]\arrow[d, hook] & X. \\
L\times J'  \arrow[ur, dashed, "?" swap] &
\end{tikzcd}
\]
Given that the vertical arrow is a $\Lambda$-anodyne extension by assumption, a lift exists.

In order to show that (3) implies (2), let $J\to J'$ be a $\Lambda$-anodyne extension, and consider the lifting problem
\[
\begin{tikzcd}
J \arrow[r]\arrow[d, hook] & X. \\
J'  \arrow[ur, dashed, "?" swap] &
\end{tikzcd}
\]
By adjointness, the lifting problem is equivalent to the following one
\[
\begin{tikzcd}
  &  \ihom{J'}{X} \arrow[d]\\
*  \arrow[ur, dashed, "?" ]\arrow[r] & \ihom{J}{X}.
\end{tikzcd}
\]
Given that $*$ is cofibrant and the vertical arrow is an acyclic fibration by assumption, a lift exists.

In order to show that (3) implies (4), and therefore (3) is equivalent to (4), let $J\to J'$ be a $\Lambda$-anodyne extension, and consider the lifting problem
\[
\begin{tikzcd}
J \arrow[r]\arrow[d, hook] &  \ihom{Y}{X}. \\
J'  \arrow[ur, dashed, "?" swap] &
\end{tikzcd}
\]
By adjointness, the lifting problem is equivalent to the following one
\[
\begin{tikzcd}
 &  \ihom{J'}{X} \arrow[d] \\
Y \arrow[ur, dashed, "?"]\arrow[r] &\ihom{J}{X}.
\end{tikzcd}
\]
Given that $Y$ (as well as any other object) is cofibrant and the vertical arrow is an acyclic fibration by assumption, a lift exists.
\end{proof}

We prove the following characterization of local objects.

\begin{prop} \label{localvscisinskiequivalences}
Let $\Lambda$ be a class of elementary anodyne extensions relative to $\mathbb{I}$. Suppose moreover that the pushout-product of any $\mathbb I$-anodyne extension $I\to J$ and any monomorphism $K\to L$,
$$(I\times L)\aamalg{I\times K}(J\times K)\to J\times L,$$
is a $\Lambda$-anodyne extension.
For a map $f\colon Y\to Y'$ in $\psh{A}$ the following are equivalent:
\begin{enumerate}
\item the map $f$ is a $\Lambda$-local weak equivalence, ie, for every $\Lambda$-local object $X$ the map 
\[
f^*\colon\ihom{Y'}{X}\to\ihom{Y}{X}
\]
is an $\mathbb I$-homotopy equivalence.
\item for every $\Lambda$-local object $X$ the map
\[
f^*\colon\ihom{Y'}{X}\to\ihom{Y}{X}
\]
induces an isomorphism on $\mathbb{I}$-homotopy classes.
\end{enumerate}
\end{prop}

\begin{rmk}
\label{homotopyhomadjunction}
We observe that the natural isomorphism
$$\Hom_{\psh{A}}(X \times Y, Z)\cong\Hom_{\psh{A}}(X, \ihom{Y}{Z}),$$
that witnesses
the adjunction between the cartesian product and the internal hom on $\psh{A}$, induces a natural bijection
$$\Hom_{\Ho_{\mathbb I}(\psh{A})}(X \times Y, Z)\cong\Hom_{\Ho_{\mathbb I}(\psh{A})}(X, \ihom{Y}{Z}),$$
at the level of homsets of the~$\mathbb{I}$-homotopy category $\Ho_{\mathbb I}(\psh{A})$ of $\psh{A}$.
\end{rmk}

\begin{proof}[Proof of \cref{localvscisinskiequivalences}]
We show that (2) implies (1). By the Yoneda Lemma, $f^*\colon\ihom{Y'}{X}\to\ihom{Y}{X}$
is an $\mathbb I$-homotopy equivalence if and only if for any $Z\in\psh{A}$ it induces a natural bijection 
$$\Hom_{\Ho_{\mathbb I}(\psh{A})}(Z, \ihom{Y}{X})\cong\Hom_{\Ho_{\mathbb I}(\psh{A})}(Z, \ihom{Y'}{X}).$$
Using \cref{homotopyhomadjunction} once the bijection becomes
$$\Hom_{\Ho_{\mathbb I}(\psh{A})}(Z\times Y,X)\cong\Hom_{\Ho_{\mathbb I}(\psh{A})}(Z\times Y',X),$$
and using it again yields
$$\Hom_{\Ho_{\mathbb I}(\psh{A})}(Y, \ihom{Z}{X})\cong\Hom_{\Ho_{\mathbb I}(\psh{A})}(Y', \ihom{Z}{X}).$$
By \cref{lemmalocalobjects}, any $X^Z$ is a $\Lambda$-local object, so it is enough to require the bijection 
$$\Hom_{\Ho_{\mathbb I}(\psh{A})}(Y, X)\cong\Hom_{\Ho_{\mathbb I}(\psh{A})}(Y', X),$$
which says precisely that $f$ induces an isomorphism when passing to $\mathbb I$-homotopy equivalence classes.
\end{proof}

\begin{rmk}
As a variant of \cite[Lemme 1.3.32]{cisinski}, it is easy to show that a map between local objects is a weak equivalence if and only if it is an $\mathbb I$-homotopy.
\end{rmk}

\subsection{Model structure on simplicial presheaves and localizers}

\label{Model structure on simplicial presheaves and localizers}

Let $A$ be a Reedy category (see eg~\cite[\textsection 15.1]{Hirschhorn}), ie, a category endowed with two subcategories $A_+$ and $A_-$ each containing all objects of $A$ and a degree function $\degr\colon\Ob(A) \to \mathbb Z_{\ge 0}$ such that
\begin{itemize}[leftmargin=*]
    \item every non-identity morphism in $A_+$ raises the degree,
    \item every non-identity morphism in $A_-$ lowers the degree, and
    \item every morphism in $A$ factors uniquely as a map in $A_-$ followed by a map in $A_+$.
\end{itemize}
We will furthermore assume that $A$ is a
``regular skeletal Reedy category'' in the sense of \cite[\textsection1.4]{Ara} or, equivalently, a ``cat\'{e}gorie squelettique r\'{e}guli\`{e}re'' in the sense of \cite[D\'efinitions 8.1.1, 8.1.36, 8.2.2]{cisinski}. 
Examples are the usual simplex category $\Delta$, and the category $t\Delta$, as will be proven in \cref{AppendixReedy}.

\begin{defn}[{\cite[\textsection1.4]{Ara}}]
\label{skeletalregular}
A Reedy category $A$ is \emph{regular skeletal} if the following conditions hold. 
\begin{enumerate}
\item Every morphism of $A_-$ admits a section.
\item Two parallel morphisms of $A_-$ are equal if and only if they admit the same set of sections.
\item Every morphism of $A_+$ is a monomorphism.
\end{enumerate}
\end{defn}

\begin{rmk}\label{regularandReedy}
A Reedy category that meets Conditions (1) and (2) is said to be ``skeletal'', also known as ``EZ-Reedy category'' 
in \cite[Definition 4.1]{elegant}. In particular, any regular skeletal Reedy category is an EZ-Reedy category, and by \cite[Proposition 4.2]{elegant} any EZ-Reedy category is an elegant Reedy category (in the sense of \cite[Definition 3.5]{elegant}).
\end{rmk}

Cofibrantly generated model structures on $\psh{A}$ with cofibrations given by the class of monomorphisms can be described in terms of their class of weak equivalences.

\begin{defn}[{\cite[\textsection2.1]{Ara}}] 
An \emph{$A$-localizer} $W$ is a class of morphisms in $\psh{A}$ such that:
\begin{enumerate}
    \item the class $W$ satisfies the $2$-out-of-$3$ property;
    \item the class $W$ contains all acyclic fibrations, ie, all maps that have the right lifting properties with respect to monomorphisms;
    \item the class of acyclic cofibrations is stable under pushouts and transfinite composition.
\end{enumerate}
For a class $C$ of maps in $\psh{A}$, there exists a smallest $A$-localizer $W(C)$ containing $C$, which we call \emph{generated} by $C$. An $A$-localizer is called \emph{accessible} if it is generated by a set.
\end{defn}

It is proven in \cite[Th\'eor\`eme 1.4.3]{cisinski}, and recalled in \cite[Theorem 2.2]{Ara}, that accessible localizers characterize the classes of weak equivalences for certain model structures on $\psh{A}$.

\begin{thm}
\label{localizersandmodelstructures}
Let $W$ be a class of morphisms in $\psh{A}$. Then $W$ is an accessible $A$-localizer if and only if it is the class of weak equivalences of a cofibrantly generated model structure on $\psh{A}$ in which the cofibrations are precisely the monomorphisms.
\end{thm}

The following remark describes the injective model structure in terms of localizers.

\begin{rmk}
\label{Ara3.7} When $A$ is a regular skeletal Reedy category (eg $A=t\Delta$), by \cite[Theorem 3.7]{Ara} the class
    \[W_{inj}:=\{f\colon X\to Y\ |\ f_a\colon X_a\xrightarrow{\simeq} Y_a\text{ for all }a\in A\}\]
of vertical levelwise weak equivalences (also considered in \cite[\textsection2.16,\textsection3.6]{Ara}) is an accessible $(A\times\Delta)$-localizer. The corresponding model structure is the injective model structure on $\spsh{A}$ where $\sset$ is endowed with the Kan-Quillen model structure, given that the cofibrations are precisely the monomorphisms.
\end{rmk}

We recall the terminology and the construction of localizations of the injective model structure.

\begin{defn}\label{definitionlocal}
Let $\Lambda$ be a set of maps of $\spsh{A}$.
\begin{itemize}
\item An object $X$ is \emph{$\Lambda$-local} if it is injectively fibrant in $\spsh{A}$
and for every map $f\colon I\to J$ in $\Lambda$, the induced map
$$f^*\colon\Map(J,X)\to\Map(I,X)$$
is a weak equivalence in $\sset$.
\item A map $g\colon C\to D$ in $\spsh{A}$ is a \emph{$\Lambda$-local weak equivalence} if for every $\Lambda$-local object $X$, the induced map
$$g^*\colon\Map(D,X)\to\Map(C,X)$$
is a weak equivalence in $\sset$.
\end{itemize}
\end{defn}

We can also describe Bousfield localizations in terms of localizers.

\begin{rmk}
\label{localization}
Let $A$ be a regular skeletal Reedy category (eg $A=t\Delta$), $W$ an accessible $A\times\Delta$-localizer and $\Lambda$ a set of maps of $\spsh{A}$. Since the localizer generated by $W$ and $\Lambda$ is accessible, 
by \cite[Proposition A.11]{Ara}
the corresponding model structure is the Bousfield localization \cite[Theorem 4.46]{barwick2010} of the model structure corresponding to $W$ with respect to the set $\Lambda$.
In particular,
\begin{itemize}
\item the cofibrations are precisely the monomorphisms, and in particular all objects are cofibrant;
\item the fibrant objects are precisely the $\Lambda$-local objects;
\item the weak equivalences are precisely the $\Lambda$-local weak equivalences;
\item the weak equivalences between fibrant objects are precisely the levelwise weak equivalences.
\end{itemize}
\end{rmk}

\section{Technical results on anodyne extensions}
\label{TechnicalResultsAnodyne}
The aim of this section is to show that the pushout-product of an elementary anodyne extension $I\to J$ and a generating monomorphism $K\to L$,
$$p\colon(I\times L)\aamalg{I\times K}(J\times K)\to J\times L,$$
is an anodyne extension, in both $\strat$ and $\psh{t\Delta}$.

\begin{defn}
\label{anodyneinstrat}
A $\Lambda_n$-anodyne extension in $\strat$, respectively $\pstrat$, is a map in $\strat$, respectively $\pstrat$, that can be written as a retract of a transfinite composition of pushouts in $\strat$, respectively $\pstrat$, of the elementary $\Lambda_n$-elementary anodyne extensions from \cref{anodynemaps}.
\end{defn}

Given that pushouts in $\strat$ and $\pstrat$ are in general different (cf~\cref{counterexamplepushout}), the meaning of the source of $p$ and the question of whether $p$ is a $\Lambda_n$-anodyne extension a priori depend on the ambient category. We show that, in fact, they do not.

\begin{prop}
\label{PPinstratandpstrat}
If $I\to J$ and $K\to L$ are two monomorphisms in $\strat$,
the pushout in $\pstrat$ of 
\[
\begin{tikzcd}
I\times L &I\times K \arrow[l, hook'] \arrow[r, hook]& J\times K,
\end{tikzcd}
\]
is a stratified simplicial set. In particular, it is also a pushout in $\strat$.
\end{prop}

\begin{proof}
The pushout-product of the two monomorphisms $I\to J$ and $K\to L$,
$$p\colon(I\times L)\aamalg{I\times K}(J\times K)\to J\times L,$$
is always a monomorphism in $\psh{A}$. Indeed, this is true in $\set$, and product constructions, pushout constructions, and the property of being a monomorphism are all checked levelwise.
This means that the prestratified simplicial set in the left-hand side is included in the stratified simplicial set $J\times L$, and it is therefore one, too.
\end{proof}

\begin{thm}
\label{anodyneinStratisanodyneinpStrat}
For $n\ge0$, any $\Lambda_n$-anodyne extension in $\strat$ is also a $\Lambda_n$-anodyne extension in $\pstrat$.
\end{thm} 

We can easily prove the theorem by means of the following lemma, which will be proven afterwards.
\begin{lem}
\label{retractlemma}
Let $A$ be a stratified simplicial set and $f\colon A \to B$ a monomorphism in $\pstrat$. Then $Rf\colon A\cong RA\to RB$ is a retract of $f$.
\end{lem}

\begin{proof}[Proof of \cref{anodyneinStratisanodyneinpStrat}]
Any anodyne extension $A \to B$ in $\strat$ is a retract of a transfinite composition of pushouts of elementary anodyne maps in $\strat$. Given that filtered colimits and retracts in $\strat$ and $\pstrat$ coincide, as mentioned in \cref{rmkretracts}, it suffices to show that the pushout in $\strat$ of an elementary $\Lambda_n$-anodyne extension is an anodyne extension in $\pstrat$.

Given any elementary $\Lambda_n$-anodyne extension $I\to J$, and any morphism $I\to X$ in $\strat$, the pushout $X\to P$ in $\psh{t\Delta}$ of $I\to J$ along $I\to X$ is a $\Lambda_n$-anodyne extension (and in particular a monomorphism) in $\psh{t\Delta}$. By \cref{retractlemma}, the map $X\cong RX\to RP$ is a retract of $X\to P$, so in particular it is a $\Lambda_n$-anodyne extension in $\psh{t\Delta}$. We conclude observing that this map is precisely the pushout of $I\to J$ along $I\to X$ in $\strat$.
\end{proof}

We now prove the lemma.

\begin{proof}[Proof of \cref{retractlemma}]
We construct a map $j\colon RB\to B$ in $\psh{t\Delta}$ as follows. First, we define $j$ on the underlying simplicial set.
\begin{itemize}
    \item[(0)] For any $m\ge0$, the component
    $$j_m\colon (RB)_m\cong B_m\to B_m$$
    is the identity.
    \end{itemize}
    
    We are left to construct a map $t(RB)_m \to tB_m$ for every $m\geq 1$. Recall that $t(RB)_m$ was defined as $\im(\varphi\colon tB_m \to B_m)$, so that we are essentially looking for a section of a surjective function. However, in order to obtain a $t\Delta$-set we need to ensure some compatibilities; that is, the (unique!) markings in $A$ are mapped identically and that the distinguished markings of degenerate simplices are mapped to such again. This leads to the following three steps.
    \begin{itemize}
    \item[(1)] For any $m\ge1$, the restriction of the $[m]_t$-component to the image of $Rf$,
    $$tj_m\colon \left(t(RB)_m\cap\im(Rf)\right)\to tB_m$$ is given by $tj_m(t(Rf)_m(a)):=tf_m(a)$ for any $a\in tA_m$. 
    \end{itemize}
    We argue that this assignment is well-defined, in that if $t(Rf)_m(a)=t(Rf)_m(a')$ for $a, a'\in tA_m$, then $a=a'$. 

    Suppose then that $t(Rf)_m(a)=t(Rf)_m(a')$. Since $f$ and $Rf$ agree on the underlying simplicial sets, we obtain the equalities 
            \begin{align*}
            f_m(\varphi(a))&=(Rf)_m(\varphi(a))=\varphi(t(Rf)_m(a))=\varphi(t(Rf)_m(a'))\\
            & = (Rf)_m(\varphi(a'))=f_m(\varphi(a'))
        \end{align*}
   Since $f$ is a monomorphism by assumption, we furthermore obtain that $a$ and $a'$ share the same underlying simplex $\varphi(a)=\varphi(a')$. Finally, since $A$ is a stratified simplicial sets, $\varphi$ is injective and $a=a'$, as claimed.
   
   Next, we define $j$ on all the marked simplices which are distinguished markings of a degenerate simplex.
    \begin{itemize}
        \item[(2)] For any $m\ge1$, the restriction of the $[m]_t$-component to the image of $\zeta_k$ for $0\leq k \leq m-1$,
    $$tj_m\colon \left(t(RB)_m\cap\im(\zeta_k)\right)\to tB_m$$ is given by $j(\zeta_k(b)):=\zeta_k(b)\in t(RB)_m$ for any $b\in B_{m-1}=(RB)_{m-1}$ and some $0\leq k \leq m-1$.
    \end{itemize}
    We need to check that simplices that are degenerate in several ways lead to the same element, and that this is compatible with the assignment on the preimage of $A$. The latter is clear, since $A$ is stratified. We now address the case of multiple degeneration, by showing that for any element of $t(RB)_m$ which can be written in the form $\zeta_k(b)$, different representatives give the same element in $t(RB)_m$. Suppose then that $\zeta_kb=\zeta_lb'$ for some $b,b' \in B_{m-1}=RB_{m-1}$. Then in particular $s_kb=s_lb'$. Assume without loss of generality that $k<l$ (the case $k=l$ may be excluded since $s_k$ is injective). Then
\[b=d_ks_kb=d_ks_lb'=s_{l-1}d_kb'\in(RB)_{m-1}.\]
From the relation $s^k\zeta^{(l-1)+1}=s^{l-1}\zeta^{k}$ for $k\leq l-1$, we obtain that
\[
\zeta_kb=\zeta_ks_{l-1}(d_kb')=\zeta_ls_kd_kb'\in\thin{B}{m}.
\]
Since $s_lb'=s_kb=s_ks_{l-1}d_kb'=s_ls_kd_kb'$ and $s_l$ is injective, we conclude $b'=s_kd_kb'$ and in particular
\[\zeta_kb=\zeta_ls_kd_kb'=\zeta_lb'\in\thin{B}{m},\]
as claimed.

Finally, we look at all the remaining marked simplices, which are just mapped to any preimage.
\begin{itemize}
        \item[(3)] For any $m\ge1$, the restriction of the $[m]_t$-component to the complement of the image of $Rf$ and $\zeta_k$'s,
    $$tj_m\colon t(RB)_m\setminus\left(\im(Rf)\cup\bigcup_k\im(\zeta_k)\right)\to tB_m$$ is given by $j(\tilde b):=b\in \thin{B}{m}$ for any $b\in\thin{B}{m}$ such that $\varphi(b)=\tilde b$ in  $\thin{RB}{m}=\im(\thin{B}{m} \xrightarrow{\varphi} B_m)$.
\end{itemize}
We conclude observing that the components of $j$ were designed to assemble to a well-defined map of prestratified simplicial sets, it is identity on underlying simplicial sets and is compatible with markings and degenerate markings, and moreover the diagram
\[
\begin{tikzcd}
A \arrow[r, equal]\arrow[d, "Rf" swap] & A\arrow[d, "f"]\\
RB \arrow[r, "j" swap]& B.
\end{tikzcd}
\]
commutes, as desired.
\end{proof}

Thanks to \cref{PPinstratandpstrat,anodyneinStratisanodyneinpStrat}, showing that the pushout-product of a generating monomorphism and an elementary anodyne extension is anodyne in $\pstrat$ boils down to showing that the same map is an anodyne extension in $\strat$.

\subsection{Pushout-products of anodyne extensions and monomorphisms in $\strat$}

In this subsection, we show that the pushout-product
of certain $\Lambda_n$-anodyne extensions $I\to J$ and a generating monomorphism $K\hookrightarrow L$ is a $\Lambda_n$-anodyne extension in $\strat$. In each of the cases that we treat, the map 
$$(I\times L)\aamalg{I\times K}(J\times K)\to J\times L$$
is an entire inclusion. This means that it is an identity on the underlying simplicial set, and we only need to show that we can mark all the simplices that are marked in $J\times L$ by means of a pushout with a suitable elementary $\Lambda_n$-anodyne extension.

\label{Pushout-product of anodyne extensions and monomorphisms}
\begin{lem}
\label{PPsaturationboundary}
Let $n\ge0$. Given $m\geq 1$ and $l\geq -1$, the pushout-product of the saturation anodyne map $\Delta[l]\star\eqDelta\hookrightarrow \Delta[l]\star\Delta[3]^\sharp$ with the boundary inclusion $\partial\Delta[m]\hookrightarrow\Delta[m]$,
\[
(\Delta[l]\star\eqDelta)\times \Delta[m] \quad \stackunder[0.2cm]{\amalg}{\scriptstyle{(\Delta[l]\star\eqDelta)\times \partial\Delta[m]}} \quad (\Delta[l]\star\Delta[3]^{\sharp})\times \partial\Delta[m]  \hookrightarrow(\Delta[l]\star\Delta[3]^{\sharp})\times \Delta[m],
\]
is a $\Lambda_n$-anodyne extension in $\strat$.
\end{lem}

The strategy for the proof was shared with us by Riehl in a personal communication.

\begin{proof}
We first observe that the putative anodyne extension is 
an entire inclusion with underlying simplicial set $(\Delta[l]\star\Delta[3])\times\Delta[m]\cong\Delta[l+4]\times\Delta[m]$, and we now analyze the differences in the stratifications.

For $k\geq l+3$, we observe that every $k$-simplex is marked in $\Delta[l]\star\eqDelta$. 
Indeed, a $k$-simplex of $\Delta[l]\star\eqDelta$ can be written in the form
$$\alpha_1\star\alpha_2\colon\Delta[k_1]\star\Delta[k_2]\cong\Delta[k_1+1+k_2]\xrightarrow{}\Delta[l]\star\Delta[3]$$
with $k_1, k_2\geq -1$ and $k_1+k_2=k-1$. Now if $k_2\geq 2$, then the resulting $k$-simplex is marked since $\eqDelta$ is $1$-trivial. If $k_2\leq 1$, we conclude $k_1\geq k-2 \geq l+1$, so that the $k_1$-simplex $\alpha_1$ of $\Delta[l]$ is necessarily degenerate and thus marked. This shows that every $k$-simplex for $k\geq l+3$ is marked in $\Delta[l]\star\eqDelta$.
Combined with the fact that $\Delta[m]$ is $m$-trivial, we obtain that every $k$-simplex of the left-hand side (and thus of the right-hand side) is marked for $k\geq \max\{l+3,m+1\}$. 
For $k\leq m$, every marked $k$-simplex of $\Delta[m]$ is degenerate, and therefore lies in $\partial \Delta[m]$, so that the marked simplices in dimensions $k\leq m$ coincide on both sides. This means the simplices that are marked in the right-hand side and not in the left-hand side are concentrated in dimensions $m < k < l+3$, whenever this set is non-empty. 

Any $k$-simplex of $\Delta[l+4]\times \Delta[m]$ can be written in the form 
$$(\alpha_1\star\alpha_2, \beta)\colon\Delta[k_1]\star\Delta[k_2]\cong \Delta[k] \xrightarrow{}(\Delta[l]\star\Delta[3])\times\Delta[m]$$
with $k_1, k_2\geq -1$ and $k_1+k_2=k-1$.
For such a simplex being marked on the right-hand side but not on the left-hand side, it is necessary that $k_2=1$ and that $\alpha_2$ represents one amongst $[01], [12], [23], [03]$. Furthermore, $\beta$ needs to be surjective, since otherwise $\beta$ would be contained in $\partial \Delta[m]$. 
Finally, we observe that if $k_1>l$, then also $\alpha_1$ needs to be degenerate and so $\alpha_1\star\alpha_2$ is already marked in $\Delta[l]\star\eqDelta$. Thus if $(\alpha_1\star\alpha_2, \beta)$ is marked only in the right-hand side, we may assume that $k_1=k-2\leq l$ and $\alpha_1$ is injective. 
Since $\beta$ is surjective, we also know that $\beta(k)=m$.

We start by marking all simplices $(\alpha_1\star\alpha_2, \beta)$ for fixed $\alpha_1$ and $\beta$ and for which $\beta(k-1)=m$, by taking a suitable pushout along an elementary anodyne extension. To this end, consider the map of simplicial sets
$$(\alpha_1\star\id, \beta\circ s^{k}\circ s^{k+1})\colon\Delta[k-2]\star\Delta[3]\cong\Delta[k+2] \to (\Delta[l]\star\Delta[3])\times \Delta[m].$$
In order to upgrade it to a map of stratified simplicial sets
\begin{align*}
&\Delta[k-2]\star\eqDelta \to (\Delta[l]\star\eqDelta)\times \Delta[m] \quad \stackunder[0.2cm]{\amalg}{\scriptstyle{(\Delta[l]\star\eqDelta)\times \partial\Delta[m] }} \quad (\Delta[l]\star\Delta[3]^{\sharp})\times \partial\Delta[m],
\end{align*}
we observe that a simplex
$$\gamma_1\star\gamma_2\colon\Delta[s_1]\star\Delta[s_2] \xrightarrow{} \Delta[k-2]\star\Delta[3]$$
with $s_1,s_2\geq -1$ is marked in $\Delta[k-2]\star\eqDelta$ if and only if at least one amongst $\gamma_1$ and $\gamma_2$ is degenerate, or if $\gamma_2=[02]$ or $\gamma_2=[13]$, or $s_2=2,3$.
\begin{itemize}
    \item In the first case, one of the components $\gamma_i$ is degenerate, and therefore so is its composite with either of $\alpha_1\star\id$ and $\beta\circ s^{k}\circ s^{k+1}$. 
    \item In the second case, the simplex
    $(\alpha_1\star\id)\circ(\gamma_1\star\gamma_2)=(\alpha_1\circ\gamma_1)\star\gamma_2$ is marked in $\Delta[l]\star\eqDelta$ by definition, and the simplex $\beta\circ s^{k}\circ s^{k+1}(\gamma_1\star\gamma_2)$ is degenerate, as a consequence of the expressions
\begin{align*}
&\beta\circ s^{k}\circ s^{k+1}\circ(\gamma_1\star[02])(s_1+1)=\beta(k-1)=m\\
&\beta\circ s^{k}\circ s^{k+1}\circ(\gamma_1\star[02])(s_1+3)=\beta(k)=m\\
&\beta\circ s^{k}\circ s^{k+1}\circ(\gamma_1\star[13])(s_1+2)=\beta(k)=m\\
&\beta\circ s^{k}\circ s^{k+1}\circ(\gamma_1\star[13])(s_1+4)=\beta(k)=m.
\end{align*}
    \item In the third case, the simplex $(\alpha_1\circ\gamma_1)\star \gamma_2$ is marked in $\Delta[l]\star\eqDelta$ by definition of $\eqDelta$, and the simplex $\beta\circ s^{k}\circ s^{k+1} \circ (\gamma_1\star \gamma_2)$ is degenerate. 
 \end{itemize}
These considerations guarantee that the map of simplicial sets defines a map of stratified simplicial sets
\begin{align*}
&(\alpha_1\star\id, \beta\circ s^{k}\circ s^{k+1})\colon\\
&\Delta[k-2]\star\eqDelta \to (\Delta[l]\star\eqDelta)\times \Delta[m] \quad \stackunder[0.2cm]{\amalg}{\scriptstyle{(\Delta[l]\star\eqDelta)\times \partial\Delta[m] }} \quad (\Delta[l]\star\Delta[3]^{\sharp})\times \partial\Delta[m].
\end{align*}

Taking the pushout of this map along the elementary saturation anodyne map
$$\Delta[k-2]\star\eqDelta\to \Delta[k-2]\star\Delta[3]^{\sharp}$$
would mark, in particular, the $k$-simplex 
\[
(\alpha_1\star\id, \beta\circ s^{k}\circ s^{k+1}) \circ(\id\star\alpha_2)=(\alpha_1\star\alpha_2, \beta)\colon\Delta[k]\to\Delta[l+4]\times \Delta[m]
\]
for $\alpha_2=[01], [12], [23], [03]$.
Taking the pushout along the sum (over $\alpha_1$ and $\beta$ such that $\beta(k-1)=m$) of all maps constructed this way, we obtain a new stratified set $P$, which is an entire sub-stratified simplicial set of $\Delta[l]\star\Delta[3]^{\sharp}\times \Delta[m]$.

By the previous analysis, the only $k$-simplices marked in $\Delta[l]\star\Delta[3]^{\sharp}\times \Delta[m]$ and not in $P$ are those of the form $$(\alpha_1\star\alpha_2, \beta)\colon\Delta[k]\to(\Delta[l]\star\Delta[3])\times \Delta[m]$$
with $\alpha_2=[01], [12], [23], [03]$, $\alpha_1$ injective, $\beta$ surjective and degenerate, and $\beta(k-1)=m-1$. 

We proceed to marking all the missing simplices, by taking a suitable pushout along an elementary complicial thinness anodyne extension. For this, consider the simplicial map
\begin{align*}
   ((\alpha_1\star\alpha_2) \circ s^{k-1}, \beta \circ s^{k})\colon \Delta[k+1] \to \Delta[l+4]\times \Delta[m].
\end{align*}
In order to upgrade it to a map of stratified  simplicial sets $\Delta^{k}[k+1]'\to P$, we record the following.
\begin{itemize}
    \item If an $s$-simplex $\gamma\colon \Delta[s] \to \Delta[k+1]$ contains $\{k-1,k,k+1\}$ in its image, the simplices $(\alpha_1\star\alpha_2 )\circ s^{k-1}\circ\gamma$ and $(\beta \circ s^{k})\circ\gamma$ are both degenerate.
    \item The image of the $(k+1)$-st face $d^{k+1}\colon\Delta[k]\to\Delta[k+1]$ can be computed by means of the expressions 
\begin{align*}
    (\alpha_1\star\alpha_2) \circ s^{k-1} \circ d^{k+1} &= (\alpha_1\star\alpha_2) \circ d^k\circ s^{k-1}, \\
    \beta \circ s^{k} \circ d^{k+1} &=\beta.
\end{align*}
In particular, the first component $(\alpha_1\star\alpha_2) \circ s^{k-1}\circ d^{k+1}$ is degenerate by construction and the second component $\beta \circ s^{k} \circ d^{k+1}$ is degenerate by assumption.
    \item The image of the $(k-1)$-st face $d^{k-1}\colon\Delta[k]\to\Delta[k+1]$ can be computed by means of the expressions 
\begin{align*}
    (\alpha_1\star\alpha_2) \circ s^{k-1} \circ d^{k-1}= \alpha_1\star\alpha_2,\\
    \beta \circ s^{k} \circ d^{k-1}=\beta \circ d^{k-1} \circ s^{k-1}.
\end{align*}
Given that the second coordinate $\beta \circ s^{k}\circ d^{k-1}$ is degenerate and fulfils
\[
\beta \circ s^{k}\circ d^{k-1}=\beta \circ d^{k-1}\circ s^{k-1}(k-1)=\beta(k)=m, 
\]
the simplex $(\alpha_1\star\alpha_2, \beta \circ d^{k-1} \circ s^{k-1})$ is one of the simplices marked in $P$ by construction of $P$.
\end{itemize}
These considerations guarantee that the map of simplicial sets defines a map of stratified simplicial sets
$$\Delta^{k}[k+1]'\to P.$$

Taking the pushout of this map along the complicial thinness extension $$\Delta^{k-1}[k+1]'\to \Delta^{k-1}[k+1]''$$
would mark precisely the $k$-th face, which is $(\alpha_1\star\alpha_2, \beta)$. 
Taking the pushout along the sum (over $\alpha_1$, $\beta$ such that $\beta(k-1)=m-1$) of all maps constructed this way, marks all the missing simplices, obtaining precisely $\Delta[l]\star\Delta[3]^{\sharp}\times \Delta[m]$ from $P$, as desired.
\end{proof}

\begin{lem}
\label{PPsaturationmarking}
Let $n\ge0$. Given $n\geq 1$ and $l\geq -1$, the pushout-product
of the saturation anodyne map $\Delta[l]\star\eqDelta\hookrightarrow \Delta[l]\star\Delta[3]^\sharp$ with the generating monomorphism $\Delta[m]\hookrightarrow\Delta[m]_t$,
\[
(\Delta[l]\star\eqDelta)\times \Delta[m]_t \quad \stackunder[0.2cm]{\amalg}{\scriptstyle{(\Delta[l]\star\eqDelta)\times \Delta[m] }} \quad (\Delta[l]\star\Delta[3]^{\sharp})\times\Delta[m]  \hookrightarrow(\Delta[l]\star\Delta[3]^{\sharp})\times \Delta[m]_t
\]
is a $\Lambda_n$-anodyne extension in $\strat$.
\end{lem}

\begin{proof}
We first observe that the putative anodyne extension is an entire inclusion with underlying simplicial set $\Delta[l+4]\times\Delta[m]$, and we now analyze the differences in the stratifications. If a $k$-simplex
$$(\sigma_1, \sigma_2)\colon \Delta[k] \to \Delta[l+4] \times \Delta[m]$$
is marked in the right-hand side, either $\sigma_2$ is an identity (and in particular $k=m$) or $\sigma_2$ is degenerate. In the latter case, the simplex is already marked in $\Delta[l]\star\Delta[3]^{\sharp}\times\Delta[m]$ and thus in the left-hand side, so without loss of generality we can assume that the $k$-simplex is of the form
$$(\sigma_1, \id) \colon \Delta[k] \to \Delta[l+4] \times \Delta[m].$$
We also know that $\sigma_1$ can be written in the form 
\[
\sigma_1=\alpha_1\star\alpha_2\colon \Delta[k_1]\star \Delta[k_2]=\Delta[k_1+1+k_2]=\Delta[m] \to \Delta[l]\star\Delta[3]
\]
for some $k_1,k_2\geq -1$ and some $\alpha_1, \alpha_2$. Since $\sigma_1$ is marked in $\Delta[l]\star\Delta[3]^{\sharp}$, either $\alpha_1$ is marked (and therefore degenerate) in $\Delta[l]$ or $\alpha_2$ is marked in $\Delta[3]^{\sharp}$ (possibly both). This implies that $\sigma_1$ is also marked in $\Delta[l]\star\eqDelta$ and $(\sigma_1, \id)$ is already marked in the left-hand side, unless $\alpha_1$ is injective, $k_2=1$ and $\alpha_2=[01], [12], [23], [13]$.

We will mark the missing simplices $(\alpha_1\star\alpha_2, \id)$ by taking a suitable pushout along a complicial thinness anodyne extension. To this end,
consider the map of simplicial sets 
\[
((\alpha_1\star\alpha_2) \circ s^{m-1}, s^m)\colon\Delta[m+1] \to \Delta[l]\star\Delta[3]\times\Delta[m].
\]
In order to upgrade it to a map of stratified simplicial sets
\[
\Delta^m[m+1]'\to \Delta[l]\star\eqDelta\times \Delta[m]_t \aamalg{\Delta[l]\star\eqDelta\times \Delta[m] }\Delta[l]\star\Delta[3]^{\sharp}\times\Delta[m],
\]
we record the following.
\begin{itemize}
\item If an $s$-simplex $\gamma\colon \Delta[s] \to \Delta[m+1]$ contains $\{m-1,m,m+1\}$ in its image, the simplices $(\alpha_1\star\alpha_2) \circ s^{m-1}\circ\gamma$ and $s^m\circ\gamma$ are both degenerate.

\item The image of the $(m-1)$-st face $d^{m-1}\colon\Delta[m]\to\Delta[m+1]$ can be computed by means of the expressions 
\[
((\alpha_1\star\alpha_2) \circ s^{m-1}, s^m) \circ d^{m-1}=(\alpha_1\star\alpha_2, s^m\circ d^{m-1}) .
\]
Given that the second component is degenerate and the first component is marked in $\Delta[l]\star\Delta[3]^{\sharp}$, the resulting simplex is marked in $\Delta[l]\star\Delta[3]^{\sharp}\times\Delta[m]$.
   \item The image of the $(m+1)$-st face $d^{m+1}\colon\Delta[m]\to\Delta[m+1]$ can be computed by means of the expressions 
\[
(\alpha_1\star\alpha_2 \circ s^{m-1}, s^m) \circ d^{m+1}=((\alpha_1\star\alpha_2)\circ s^{m-1} \circ d^{m+1}, \id) 
\]
Given that the first component is degenerate and the second is an identity, this simplex is marked in $\Delta[l]\star\eqDelta\times \Delta[m]_t$.
\end{itemize}
These considerations guarantee that the map of simplicial sets defines a map of stratified simplicial sets 
\[
\Delta^m[m+1]'\to \Delta[l]\star\eqDelta\times \Delta[m]_t \aamalg{\Delta[l]\star\eqDelta\times \Delta[m] }\Delta[l]\star\Delta[3]^{\sharp}\times\Delta[m].
\]
Taking the pushout of this map along the complicial thinness extension $$\Delta^m[m+1]' \to \Delta^m[m+1]''$$
would mark the simplex
$(\alpha_1\star\alpha_2, \id)$.
Taking the pushout of the sum (over $\alpha_1$ and $\alpha_2$) of the maps constructed in this way along the complicial thinness extension $\Delta^m[m+1]' \to \Delta^m[m+1]''$ marks all simplices $(\alpha_1\star\alpha_2, \id)$, as desired.
\end{proof}

\begin{lem}
\label{PPtrivialityboundary}
Let $n\ge0$. Given $m \geq 0$ and $l> n$, the pushout-product of the triviality anodyne map $\Delta[l] \hookrightarrow \Delta[l]_t$ with the boundary inclusion $\partial\Delta[m]\hookrightarrow\Delta[m]$,
\[
\Delta[l]\times \Delta[m] \aamalg{\Delta[l]\times \partial\Delta[m] }\Delta[l]_t\times \partial\Delta[m]  \hookrightarrow \Delta[l]_t\times \Delta[m],
\]
is a $\Lambda_n$-anodyne extension in $\strat$.
\end{lem}

\begin{proof}
We first observe that the putative anodyne extension is an entire inclusion with underlying simplicial set $\Delta[l]\times\Delta[m]$, and we now analyze the differences in the stratifications.

If a $k$-simplex
$$(\sigma_1, \sigma_2)\colon \Delta[k] \to \Delta[l] \times \Delta[m]$$
is marked in the right-hand side, then $\sigma_1$ needs to be either degenerate or an identity, and $\sigma_2$ needs to be degenerate. If $\sigma_1$ is degenerate, then $(\sigma_1, \sigma_2)$ is already marked in $\Delta[l]\times \Delta[m]$ and thus in the left-hand side. So without loss of generality we can assume that $\sigma_1=\id$, and in particular $k=l$. This means that the only simplices that are marked in the right-hand side and not in the left-hand side must be of dimension $l$.
These simplices can be marked by taking a pushout of a sum of the triviality elementary anodyne extension $\Delta[l] \to \Delta[l]_t$.
\end{proof}

\begin{lem}
\label{PPtrivialitymarking}
Let $n\ge0$. Given $m\geq 1$ and $l>n$, the pushout-product
of the triviality anodyne map $\Delta[l]\hookrightarrow\Delta[l]_t$ with the generating monomorphism $\Delta[m]\hookrightarrow\Delta[m]_t$,
\[
\Delta[l]\times \Delta[m]_t \aamalg{\Delta[l]\times \Delta[m] }\Delta[l]_t\times\Delta[m]  \hookrightarrow \Delta[l]_t\times \Delta[m]_t
\]
is a $\Lambda_n$-anodyne extension in $\strat$.
\end{lem}

\begin{proof}

We first observe that the putative anodyne extension is an entire inclusion with underlying simplicial set $\Delta[l]\times\Delta[m]$, and we now analyze the differences in the stratifications.

If a $k$-simplex
$$(\sigma_1, \sigma_2)\colon \Delta[k] \to \Delta[l] \times \Delta[m]$$
is marked in the right-hand side, then $\sigma_1$ and $\sigma_2$ need to be either degenerate or identities. If at least one of the two is degenerate, then $(\sigma_1, \sigma_2)$ is already marked in the left-hand side. Thus, without loss of generality we can assume that $k=l=n$ and $\sigma_1=\sigma_2=\id$.
This means that the only simplices that are marked in the right-hand side and not in the left-hand side must be of dimension $l$.
These simplices can be marked by taking a pushout of a sum of the triviality elementary anodyne extension $\Delta[l] \to \Delta[l]_t$.
\end{proof}

\section{The Reedy structure on $t\Delta$}
\label{AppendixReedy}

We aim to endow $t\Delta$ with a Reedy structure. 
We refer the reader to \cite[\textsection 15.1]{Hirschhorn} for the Reedy structure $(\Delta_+,\Delta_-)$ on $\Delta$, with respect to the subcategory $\Delta_+$ consisting of all injective maps and the subcategory $\Delta_-$ consisting of all surjective maps.

\begin{notn}
\label{reedystructure}
Let $\degr\colon\Ob(t\Delta) \to \mathbb Z_{\ge 0}$ be the degree function defined by
\begin{align*}
&\degr([0])=0,\\
&\degr([k])=2k-1 \mbox{ for } k\geq 1,\\
&\degr([k]_t)=2k\mbox{ for } k\geq 1.
\end{align*}
The resulting (pre)order on $t\Delta$ can be pictured as
$$[0]<[1]<[1]_t<\dots<[k]<[k]_t<\dots.$$
We denote
\begin{itemize}
\item by $t\Delta_+$ the subcategory of $t\Delta$ generated by $\Delta_+$
and by the comarking maps $\varphi \colon [k]\to[k]_t$ for all $k\geq 1$, and
\item by $t\Delta_-$ the subcategory of $t\Delta$ generated by $\Delta_-$
and by the maps $\zeta^i \colon [k]_t\to[k-1]$ for all $k\geq 1$ and all $0\leq i \leq k-1$.
\end{itemize}
These subcategories are well-defined since $\Delta$ is a full subcategory of $t\Delta$ (cf~\cref{DeltaintDelta}).
\end{notn}

\begin{prop}\label{tDeltaReedy}
The category $t\Delta$ endowed with the structure from \cref{reedystructure} is a Reedy category.
\end{prop}

We need a preliminary lemma.

\begin{lem}
\label{varphiEpi}
For every $k\ge1$, the map $\varphi\colon[k]\to[k]_t$ is a monomorphism and an epimorphism in $t\Delta$. 
\end{lem}

\begin{proof}[Proof of \cref{varphiEpi}]
To see that $\varphi\colon[k]\to[k]_t$ is a monomorphism, it is enough to observe that there is no relation in $t\Delta$ of the form $\varphi\circ\alpha=\alpha'$.

In order to show that $\varphi\colon[k]\to[k]_t$ is an epimorphism, let two maps $\alpha, \beta \colon [k]_t\to [m]_{(t)}$ be given so that $\alpha\varphi=\beta\varphi$. Since we already know $\varphi$ is a monomorphism in $t\Delta$, we can assume that the target is actually $[m]$. Since only generating morphisms in $t\Delta$ starting in $[k]_t$ are the $\zeta^i$'s and since $\Delta$ is a full subcategory, we can find $\alpha', \beta'\in \Delta$ as well as $0\leq i,j\leq k-1$ so that $\alpha'\zeta^i=\alpha$ and $\beta=\beta'\zeta^j$. Thus, $\alpha\varphi=\beta\varphi$ implies $\alpha' s^i= \beta' s^j$. 
 Assume without loss of generality that $i\leq j$. 
If $i=j$, then precompose with $d^i$ to arrive at $\alpha'=\beta'$, and we are done in this case.

If $i<j$, then the relation $\alpha' s^i = \beta' s^j$
 implies by precomposing with $d^i$ that $\alpha'=\beta' s^j d^i =\beta' d^i s^{j-1}$. Precomposing with $d^j$ yields
$\beta'= \alpha' s^i d^j=\alpha' d^{j-1} s^i$ for $i< j-1$. Now we employ the relation $s^{i}\zeta^j=s^{j-1}\zeta^i$ to arrive at 
\[
\alpha'\zeta^i=\beta' d^i s^{j-1}\zeta^i=\beta' d^i s^i \zeta^j=\alpha' d^{j-1} s^i d^i s^i\zeta^j=\alpha' d^{j-1} s^i\zeta^j=\beta' \zeta^j,
\]
as desired. 
For $i=j-1$, we get $\alpha'=\beta'$ by precomposing with $d^{j}=d^{i+1}$ additionally to the already obtained $\alpha'=\beta'd^is^i$. Thus, we obtain, employing $s^{i}\zeta^{i+1}=s^{i}\zeta^i$:
\[
\alpha'\zeta^i=\alpha'd^is^i\zeta^i=\alpha'd^is^{i}\zeta^{i+1}=\beta'\zeta^j,
\]
which completes the proof. 
\end{proof}

We can now prove the proposition.

\begin{proof}[Proof of \cref{tDeltaReedy}]
As a preliminary remark, we observe that in $t\Delta_{-}$, there are no generating maps (and thus no non-identity maps) whose target is in $t\Delta\setminus\Delta$. Similarly, there are no non-identity maps in $t\Delta_{+}$ whose source is in $t\Delta\setminus\Delta$. 
Since $\varphi$ is the only generator with target in $[n]_t$, any map in $t\Delta_+$ either is in $\Delta_+$ or there exists a decomposition
\[
\begin{tikzcd}
{[m]} \arrow[r, tail, "\beta"] & {[n]}\arrow[r, tail, "\varphi"]& {[n]_t}
\end{tikzcd}
\]
with $\beta \in \Delta_+$.
Since $\varphi$ is a monomorphism by \cref{varphiEpi}, this decomposition is unique. 
Similarly, each map in $t\Delta_{-}$ is either in $\Delta_{-}$ or can be (non-uniquely!) decomposed as 
\[
\begin{tikzcd}
{[k]_t} \arrow[r, two heads, "\zeta_i"] & {[k-1]}\arrow[r, two heads, "\beta"]& {[n]}
\end{tikzcd}
\]
with $\beta \in \Delta_-$.

We need to show that each map in $t\Delta$ can be uniquely decomposed into a map in $t\Delta_{-}$, followed by a map in $t\Delta_{+}$. For this, we distinguish several cases. In every case, we first provide a factorization and then prove its uniqueness.
\begin{enumerate}
    \item A map $\alpha\colon[m]\to[n]$ between elements of $\Delta$ factors as
$$[m]\twoheadrightarrow[\im(\alpha)]\rightarrowtail[n],$$
where $\im(\alpha)+1$ is the cardinality of the image of $\alpha$.
Since $\Delta$ is Reedy, this factorization is unique in $\Delta$, and a different factorization in $t\Delta$ would need to be of the form
\[
[m] \twoheadrightarrow [k]_t \rightarrowtail [n],
\]
which is impossible as there are no maps of this form in $t\Delta_{-}$. 
    \item A map of the form $\alpha\colon[m]\to[n]_t$ factors as
$$[m]\twoheadrightarrow[k]\rightarrowtail[n]\rightarrowtail[n]_t,\quad$$
where the first two maps form the canonical factorization in $\Delta$ as in the first case.
Given any other factorization 
\[
[m]\twoheadrightarrow[k']_{(t)}\rightarrowtail[n]_t,
\]
we know by the preliminary remark that the latter map needs to be of the form $[k'] \rightarrowtail [n] \rightarrowtail [n]_t$, with the first map $[k'] \rightarrowtail [n]$ being in $\Delta_+$. This reduces the uniqueness of the factorization again to the analogous result in $\Delta$.
    \item A map of the form $\alpha\colon[m]_t\to[n]$ factors as
\[
\begin{tikzcd}
{[m]_t}\arrow[r, two heads, "\zeta^i"] &{[m-1]} \arrow[r, two heads, "\beta"] & {[k]} \arrow[r, tail] & {[n]},
\end{tikzcd}
\]
where the last two maps come from the factorization in $\Delta$. The choice of $i$ might not be unique, but we claim that the composite of $\zeta^i$ and $\beta$ is unique. Given any other factorization
\[
\begin{tikzcd}
{[m]_t}\arrow[r, two heads]  & {[l]_{(t)}} \arrow[r, tail] & {[n]},
\end{tikzcd}
\]
 we can conclude that the intermediate object has to be in $\Delta$ by the preliminary remark. 
 Now we can precompose both factorizations with $\varphi\colon[m]\to[m]_t$. The resulting factorizations
 \[
\begin{tikzcd}
{[m]}\arrow[r, two heads]  & {[l]} \arrow[r, tail] & {[n]}
\end{tikzcd}\text{ and }
\begin{tikzcd}
{[m]}\arrow[r, two heads]  & {[k]} \arrow[r, tail] & {[n]}
\end{tikzcd}
\]
 have to coincide since $\Delta$ is a Reedy category. In particular, the maps
  \[
\begin{tikzcd}
{[l]} \arrow[r, tail] & {[n]}
\end{tikzcd}\text{ and }
\begin{tikzcd}
{[k]} \arrow[r, tail] & {[n]}
\end{tikzcd}
\]
 coincide, and the maps
  \[
\begin{tikzcd}
{[m]}\arrow[r, two heads]  & {[l]}
\end{tikzcd}\text{ and }
\begin{tikzcd}
{[m]}\arrow[r, two heads]  & {[k]}
\end{tikzcd}
\]
 coincide.
 Moreover, by \cref{varphiEpi} we know that $\varphi\colon[m]\to[m]_t$ is an epimorphism and therefore the maps
   \[
\begin{tikzcd}
{[m]_t}\arrow[r, two heads]  & {[l]}
\end{tikzcd}\text{ and }
\begin{tikzcd}
{[m]_t}\arrow[r, two heads]  & {[k]}
\end{tikzcd}
\]
coincide.
    \item A non-identity map $\alpha\colon[m]_t\to [n]_t$ factors as
\[
\begin{tikzcd}
{[m]_t}\arrow[r, two heads, "\zeta^i"] &{[m-1]} \arrow[r, two heads, "\beta"] & {[k]} \arrow[r, tail] & {[n]}\arrow[r, tail, "\varphi"] & {[n]_t},
\end{tikzcd}
\]
where the two middle maps come from the factorization in $\Delta$.
Observe once again that since there is no relation of the form $\varphi\alpha=\alpha'$, any two different factorizations would need to agree before the last application of $\varphi$, which reduces this case to the previous one. 
\qedhere
\end{enumerate}
\end{proof}

Recall from \cref{skeletalregular} the notion of a regular skeletal Reedy category.

\begin{prop}
\label{tDeltasuperReedy}
The category $t\Delta$ is a regular skeletal Reedy category.
\end{prop}
\begin{proof}
We verify the conditions (1)-(3) from \cref{skeletalregular} for $t\Delta$.
\begin{enumerate}
    \item We need to show that every morphism in $t\Delta_-$ admits a section. First, every morphism in $\Delta_-$ admits a section, given that every morphism  in $\Delta_-$ is a composition of $s^j$ for varying $j$, and $s^jd^j=\id$. 
Next, every $\zeta^i$ admits a section since we have $\zeta^i\varphi=s^i$ and thus 
$\zeta^i\varphi d^i=s^i d^i=\id$.
Therefore, any morphism in $t\Delta_{-}$, which is a composite of maps in $\Delta_-$ and of maps of the form $\zeta^i$, admits a section, as desired.
    \item We need to show that two parallel arrows in $t\Delta_-$ coincide if they have the same set of sections.
    Note that any morphism in $t\Delta_-$ is either in $\Delta_-$ or of the form $\alpha\zeta^i$ with $\alpha \in \Delta_-$.
    
Given that $t\Delta$ does not have non-trivial automorphisms, and that $\Delta$ is a full subcategory,
the statement for the morphisms in $\Delta_{-}$ follows from the one for $\Delta$.

Now assume that $\alpha\zeta^i, \beta\zeta^j\colon [k]_t\to[l]$ have the same set of sections. Since the sections are maps of the form $[l]\to [k]_t$, all of them can be necessarily written as $\varphi \gamma$ for some $\gamma$ in $\Delta$. Since $\varphi$ is a monomorphism, 
we conclude that  $\alpha\zeta^i\varphi=\alpha s^i$ and $\beta\zeta^j\varphi=\beta s^j$ have the same set of sections. Given that $\Delta$ is regular skeletal, we obtain that
$\alpha\zeta^i\varphi=\beta\zeta^j\varphi$.
Since $\varphi$ is an epimorphism by \cref{varphiEpi}, we conclude that $\alpha\zeta^i=\beta\zeta^j$, as desired.
    \item We need to show that every morphism in $t\Delta_+$ is a monomorphism. This is true for morphisms in $\Delta_+$, given that they have left inverses, and we observed that $\varphi$ is a monomorphism before.\qedhere
\end{enumerate}

\end{proof}


\bibliographystyle{amsalpha}
\bibliography{ref}
\end{document}